\documentclass[11pt]{article}

\usepackage{amsmath,amssymb,amsthm,bm}
\usepackage[margin=1in]{geometry}
\usepackage{booktabs}
\usepackage{tikz}
\usetikzlibrary{arrows.meta}
\usepackage{xcolor}
\usepackage{hyperref}
\usepackage{comment}

\newcommand{\revised}[1]{#1}

\newcommand{\revisedblock}{}

\newtheorem{theorem}{Theorem}
\newtheorem{proposition}{Proposition}
\newtheorem{lemma}{Lemma}
\newtheorem{corollary}{Corollary}
\theoremstyle{definition}
\newtheorem{example}{Example}
\theoremstyle{remark}
\newtheorem{remark}{Remark}

\newcommand{\R}{\mathbb{R}}
\newcommand{\E}{\mathbb{E}}
\newcommand{\KL}{\mathrm{KL}}
\newcommand{\dom}{\operatorname{dom}}
\newcommand{\ran}{\operatorname{ran}}
\newcommand{\ELBO}{\mathrm{ELBO}}
\newcommand{\inner}[2]{\langle #1,\, #2 \rangle}

\title{A Tutorial on Bregman Projection in Statistics}
\author{Gunhee Cho \and Jae Kwang Kim \and Yumou Qiu }
\date{}

\begin{document}
\maketitle

\begin{abstract}
A single geometric operation --- projecting a reference onto a
constrained family under a Bregman divergence --- underlies a striking
range of statistical methods. This tutorial develops the operation first
as pure convex geometry, with no statistics attached. A strictly convex
generator $G$ and its conjugate $F$ furnish two coordinate systems, a
projection theorem with existence and uniqueness, and a Pythagorean
\revised{theorem}; the Pythagorean theorem itself produces \emph{two} dual
projections --- the information (e-) projection onto moment-constrained
families and the moment (m-) projection onto exponential families ---
exchanged by the conjugacy $G\leftrightarrow F$, so a single theorem
governs both. Part~II reads off the statistics. The generalized linear
model is treated in detail as the concrete carrier of the two
projections: \revised{under the canonical link,} the score equation is exactly the Pythagorean
orthogonality, and the fit is simultaneously an e-projection in the
natural coordinate and an m-projection in the mean coordinate. Maximum
entropy, survey calibration, over-identified moment models, the EM
algorithm, variational inference,
autoencoders, and expectation propagation then fall into place as
instances of the same construction --- exactly where the underlying
families are flat, and as controlled approximations or
neighboring-divergence analogies where they are not. The mathematics of Part~I is self-contained; the statistical
sections presume only familiarity with the methods being unified.
\end{abstract}

\tableofcontents

\part{Bregman projection}
\label{part:math}

\section{Introduction}

Fix a strictly convex function $G$ and use it to measure how far one
point lies from another by the amount $G$ rises above its own tangent
line --- the \emph{Bregman divergence} $D_G$. Given a reference point and
a constraint set, the point of the set closest to the reference in $D_G$
is the \emph{Bregman projection} of the reference onto the set. This
single construction, together with one Pythagorean identity, is the
entire content of Part~\ref{part:math}; Part~\ref{part:stat} shows that a
large fraction of statistical estimation is this construction in
disguise.

We develop the geometry with no statistics attached, because the cleanest
way to see why so many methods coincide is to first see the object they
share. Two features of the projection do all the later work. First, the
projected point has a closed form --- a generalized exponential family
generated by $G$. Second, it satisfies a Pythagorean decomposition, and
the Pythagorean theorem itself produces \emph{two} projections, dual to
one another, according to whether one varies the first or the second
argument of the divergence: the \emph{e-projection} onto a
moment-constrained ($m$-flat) family and the \emph{m-projection} onto an
exponential ($e$-flat) family. The two are exchanged by the conjugacy
$G\leftrightarrow F$, so one theorem covers both directions across all
generators \cite{Csiszar1975,AmariNagaoka,Zhang2004}.

Only once these are in hand do we turn to statistics
(Part~\ref{part:stat}). There we lead with the generalized linear model,
because it is the setting in which the two projections are most concrete
and most familiar: the negative log-likelihood is a Bregman divergence to
the saturated model, the score equation is the Pythagorean orthogonality,
and the maximum-likelihood fit is an e-projection in the natural
coordinate and, simultaneously, an m-projection in the mean coordinate.
With the two projections thus made tangible, maximum entropy, calibration,
EM, variational inference, autoencoders, and expectation propagation
follow as instances of the same construction --- exact projections where
the families are flat, controlled approximations or
neighboring-divergence analogies where they are not.

\section{Bregman divergence and dual coordinates}
\label{sec:prelim}

\subsection{The separable Bregman divergence}
Let $(\Xi,\mathcal{A},\mu)$ be a $\sigma$-finite measure space and let
$G:I\to\R$ be a {strictly convex} \emph{generator} on an open interval
$I=\dom G\subseteq\R$. The pointwise Bregman distance is the gap between
$G$ and its tangent at $t$,
\begin{equation}
	\label{eq:pointwise}
	d_G(s,t)=G(s)-G(t)-G'(t)(s-t)\ \ge 0,
\end{equation}
nonnegative by convexity and zero only at $s=t$ by strict convexity. For
measurable maps $p,q:\Xi\to I$ the \emph{separable} Bregman divergence
sums these pointwise,
\begin{equation}
	\label{eq:sepbregman}
	D_G(p\,\|\,q)=\int_\Xi d_G\big(p(\xi),q(\xi)\big)\,d\mu(\xi).
\end{equation}
The associated generalized entropy is
\[
H_G(p)=-\int_\Xi G(p)\,d\mu,
\]
up to an affine term. For the Shannon generator $G(t)=t\log t$ this is
the entropy
\[
H(p)=-\int_\Xi p\log p\,d\mu=-\E_p[\log p],
\]
written $H(\cdot)$. The divergence is convex in its first argument,
generally asymmetric, and need not obey a triangle inequality.

\begin{example}[Three generators]
	\label{ex:gen}
	$G(t)=\tfrac12 t^2$ gives $d_G(s,t)=\tfrac12(s-t)^2$, squared distance.
	$G(t)=t\log t$ gives
	\[
	d_G(s,t)=s\log\frac{s}{t}-s+t,
	\]
	which {integrates to the Kullback--Leibler divergence when
		$p$ and $q$ are probability distributions, since the linear terms integrate
		to zero}. The power/Tsallis generator
	\[
	G(t)=\frac{t^\alpha-t}{\alpha(\alpha-1)}
	\]
	{for $t>0$ and $\alpha\neq 0,1$} interpolates between common
	entropy geometries.
\end{example}

\begin{remark}[Separable vs.\ mean-parameter Bregman]
	\label{rem:separable}
	The object lifted to the infinite-dimensional, distributional space is the
	\emph{separable} divergence \eqref{eq:sepbregman} --- the U-divergence
	family of \cite{Eguchi} --- not the finite-dimensional Bregman divergence
	on a mean-parameter manifold. Every statement of Part~\ref{part:math} is
	at the separable level. {This distinction matters because a
		single generator may govern the separable projection geometry without
		automatically giving the same finite-dimensional mean-parameter geometry
		in every statistical model.}
\end{remark}

\subsection{The conjugate and two coordinate systems}
Assume throughout that $G\in C^2(I)$ is strictly convex {and
	satisfies $G''>0$}. Then $G'$ is a strictly increasing
homeomorphism from $I$ onto an open interval
\[
I^*:=G'(I),
\]
and the convex conjugate
\[
F(u)=\sup_{s\in I}\{us-G(s)\}
\]
is $C^1$ and strictly convex on $I^*$. The pair $(G,F)$ satisfies the
Fenchel--Young equality
\[
G(s)+F(u)=us
\]
exactly when $u=G'(s)$, and the two derivatives are mutually inverse:
\begin{equation}
	\label{eq:links}
	\theta=\nabla G(p)=G'(p)\quad(\text{\emph{link}}),
	\qquad
	p=\nabla F(\theta)=(G')^{-1}(\theta)\quad(\text{\emph{inverse link}}).
\end{equation}
We call $p$ the \emph{m-coordinate} --- the point itself, or mixture /
expectation coordinate --- and $\theta=\nabla G(p)$ the
\emph{e-coordinate} --- the natural or exponential coordinate.
{These names are used in the generalized Bregman sense; for
	regular exponential families they coincide with the usual mean and natural
	coordinates.} The link translates between the two coordinate systems and,
as the next lemma shows, conjugates the divergence.

\subsection{Two flat structures and two identities}

A family is \emph{m-flat} if it is affine in the m-coordinate. Its
prototype is a \emph{moment set}
\begin{equation}
	\mathcal{M}(c)
	=
	\Big\{
	p:\Xi\to I:
	\E_p[T]:=\int_\Xi p(\xi)T(\xi)\,d\mu(\xi)=c
	\Big\}, 
\label{eq:m_c}
\end{equation}
where $T:\Xi\to\R^k$ is a statistic and $c\in\R^k$ is the target moment.
{When the target is fixed and unambiguous we sometimes write
	$\mathcal{M}$, but $\mathcal{M}(c)$ is the preferred notation whenever
	several moment sets appear.}

A family is \emph{e-flat} if it is affine in the e-coordinate
$\theta=\nabla G(p)$. Its prototype is the \emph{$G$-exponential family}
through a reference point $p_0$:
\[
{
	\mathcal{E}_G(p_0,T)
	=
	\Big\{
	\nabla F\big(\nabla G(p_0)+\inner{\lambda}{T}\big):
	\lambda\in\R^k,\ 
	\nabla G(p_0)+\inner{\lambda}{T}\in I^*
	\ \mu\text{-a.e.}
	\Big\}.
}
\]
{When no confusion is possible this family will be denoted
	simply by $\mathcal{E}$, but the expanded notation
	$\mathcal{E}_G(p_0,T)$ records its dependence on the generator, reference,
	and statistics.} The link $\nabla G$ sends primal affine descriptions to
dual affine descriptions after passing to the conjugate generator $F$.
The $G$-exponential family above is a regular exponential family in the
natural coordinate $\theta$ {in the regular finite-dimensional
	setting}. The correspondence between regular exponential families and
regular Bregman divergences is exact at the mean-parameter level: the
cumulant's conjugate serves as the generator \cite{Banerjee2005}.
{This is the finite-dimensional counterpart of, but not
	identical to, the separable divergence of Remark~\ref{rem:separable}.}
The statistical sections specialize this dictionary.

\begin{lemma}[Dual-divergence identity]
	\label{lem:dualdiv}
	For measurable $p,q:\Xi\to I$,
	\begin{equation}
		\label{eq:dualdiv}
		D_G(p\,\|\,q)=D_F\big(\nabla G(q)\,\|\,\nabla G(p)\big).
	\end{equation}
\end{lemma}

\begin{proof}
	It suffices to prove the pointwise identity
	\[
	d_G(s,t)=d_F\big(G'(t),G'(s)\big).
	\]
	Using $F(G'(r))=G'(r)r-G(r)$ and $F'(G'(s))=s$,
	\[
	\begin{aligned}
		d_F\big(G'(t),G'(s)\big)
		&=
		\big[G'(t)t-G(t)\big]
		-
		\big[G'(s)s-G(s)\big]
		-
		s\big(G'(t)-G'(s)\big) \\
		&=
		G(s)-G(t)-G'(t)(s-t)
		=
		d_G(s,t).
	\end{aligned}
	\]
	Integrating over $\Xi$ gives \eqref{eq:dualdiv}.
\end{proof}

Reading a $G$-divergence in e-coordinates turns it into an $F$-divergence
with the arguments swapped; this single fact makes the two projections of
\S\ref{sec:projections} conjugate.

\begin{lemma}[Three-point identity]
	\label{lem:threepoint}
	For all $a,b,c\in I$,
	\begin{equation}
		\label{eq:threepoint}
		d_G(a,b)
		=
		d_G(a,c)+d_G(c,b)+\big(G'(c)-G'(b)\big)(a-c).
	\end{equation}
\end{lemma}

\begin{proof}
	Expand the right-hand side. The terms involving $G'(b)$ recombine into
	$-G'(b)(a-b)$, leaving
	\[
	G(a)-G(b)-G'(b)(a-b)=d_G(a,b).
	\]
\end{proof}

The cross term in \eqref{eq:threepoint} is precisely what must vanish for
a Pythagorean decomposition. The projection of the next section is the
choice of the intermediate point $c$ that kills it.

\subsection{A minimal information-geometric reading}

	The preceding construction can be read as a minimal piece of information
	geometry \cite{AmariNagaoka}. The point is not to introduce additional
	differential-geometric machinery, but to record the geometric meaning of
	the objects already defined above. A strictly convex generator $G$
	provides more than a divergence: it determines a local quadratic geometry
	and a dual way of measuring displacement.

	In finite dimension, suppose that $\Omega\subset\R^d$ is open and convex,
	and that $G\in C^2(\Omega)$ has positive definite Hessian. Then the
	second-order expansion of the Bregman divergence at the diagonal gives
	\[
	D_G(p+h\|p)
	=
	\frac12 h^\top \nabla^2G(p)h
	+
	O(\|h\|^3).
	\]
	Thus, although $D_G$ is generally asymmetric and is not a metric distance,
	its infinitesimal second-order part defines the Hessian metric
	\[
	g_p(u,v)
	=
	u^\top\nabla^2G(p)v.
	\]
	This is the local quadratic geometry induced by $G$.

	For the separable divergence used in this paper, the same calculation
	gives
	\[
	g_p(u,v)
	=
	\int_\Xi G''(p(\xi))u(\xi)v(\xi)\,d\mu(\xi),
	\]
	for tangent perturbations $u,v$ for which the integral is finite and
	$p+\varepsilon u$, $p+\varepsilon v$ remain in $I$ for sufficiently small
	$\varepsilon$. Hence the curvature $G''$ determines how local
	perturbations are weighted. In the Shannon case $G(t)=t\log t$, this
	becomes
	\[
	g_p(u,v)
	=
	\int_\Xi
	\frac{u(\xi)v(\xi)}{p(\xi)}\,d\mu(\xi),
	\]
	the Fisher quadratic form on positive densities. On normalized densities,
	the tangent vectors additionally satisfy
	\[
	\int_\Xi u\,d\mu
	=
	\int_\Xi v\,d\mu
	=
	0.
	\]

{
	The dual coordinate also has a local metric interpretation. Since
	\[
	\theta=\nabla G(p),
	\qquad
	p=\nabla F(\theta),
	\]
	one has, whenever the Hessians are invertible,
	\[
	\nabla^2F(\theta)
	=
	\bigl(\nabla^2G(p)\bigr)^{-1}.
	\]
	Thus the primal and dual descriptions carry inverse Hessian quadratic
	forms. This is the infinitesimal version of the conjugacy already
	expressed globally by the dual-divergence identity.
}

	The Pythagorean identity can now be read as a global version of this local
	orthogonality. At a projection point $p^\ast$, the dual displacement
	$\nabla G(p^\ast)-\nabla G(p_0)$ lies in the span of the constraint
	functions, while feasible perturbations $p-p^\ast$ preserve the imposed
	moments. Therefore the cross term in the three-point identity vanishes:
	\[
	\int_\Xi
	\bigl(
	G'(p^\ast(\xi))-G'(p_0(\xi))
	\bigr)
	\bigl(
	p(\xi)-p^\ast(\xi)
	\bigr)
	\,d\mu(\xi)
	=
	0.
	\]
	Thus Bregman projection is governed by orthogonality in the dual pairing
	induced by $G$, not necessarily by ordinary Euclidean orthogonality.

	In summary, the information-geometric content needed below is the
	following: the generator $G$ defines a divergence, its Hessian defines the
	local quadratic geometry, the map $\nabla G$ provides the dual coordinate,
	and the Pythagorean theorem expresses projection as orthogonality in this
	dual pairing. The later statistical examples use precisely these
	ingredients: dual coordinates, projection, and the Pythagorean identity.

\section{The projection theorem and the two projections}
\label{sec:projections}

We project a reference $p_0:\Xi\to I$ onto the moment set
{$\mathcal{M}(c)$} in (\ref{eq:m_c}) by minimizing the divergence in its
\emph{first} argument:
\begin{equation}
	\label{eq:primal}
	\min_{p\in{\mathcal{M}(c)}}\ D_G(p\,\|\,p_0) .
\end{equation}
A normalization $\int p\,d\mu=m$ is the constraint attached to
$T_0\equiv1$, $c_0=m$, and is absorbed into $T$.

\subsection*{Standing assumptions}
\begin{description}
	\item[(A1)] {$G\in C^2(I)$ is strictly convex and $G''>0$} on the open interval $I$,
	so $\nabla F=(G')^{-1}$ exists on $I^*=\ran G'$.
	
	\item[(A2)] The feasible set
	\[
	{
		\mathcal{M}(c)
		=
		\{p:\Xi\to I:\E_p[T]=c,\ D_G(p\,\|\,p_0)<\infty\}
	}
	\]
	is nonempty. {It is convex because the moment constraint is affine and
		$I$ is an open interval.}
	
	\item[(A3)] \emph{Constraint qualification.} Writing the effective dual
	domain
	\[
	\mathcal{D}
	=
	\Big\{
	\lambda\in\R^k:
	\nabla G(p_0)+\inner{\lambda}{T}\in I^*\ \mu\text{-a.e.},\
	\textstyle\int_\Xi
	F\big(\nabla G(p_0)+\inner{\lambda}{T}\big)\,d\mu<\infty
	\Big\},
	\]
	the dual objective {$\Phi$} of \eqref{eq:dual} attains its supremum at a point
	$\lambda^*\in\operatorname{ri}(\mathcal{D})$ modulo $\mathcal{N}$.
	A sufficient condition is that {$\Phi$ is coercive modulo
		$\mathcal{N}$, as happens under the usual steepness and integrability
		assumptions.} \revised{We assume in addition that $0\in\mathcal{D}$ ---
	membership at $\lambda=0$ is automatic, so this says
	$\int_\Xi F\big(\nabla G(p_0)\big)\,d\mu<\infty$ --- which makes the
	constant $C_0$ in \eqref{eq:dual} finite and the dual objective well
	defined.}

\item[(A4)] \emph{Integrability.} Differentiation under the integral sign
is valid \revised{on the relative interior of $\mathcal{D}$, so that
$\nabla \Phi(\lambda)=c-\E_{p^{*}_\lambda}[T]$ holds for every
$\lambda\in\operatorname{ri}(\mathcal{D})$ --- in particular, by (A3), in
a neighborhood of $\lambda^{*}$.}
  
\end{description}

\begin{remark}[Where (A3) and (A4) bite]
\label{rem:a4}
For the Shannon generator, $I^{*}=\R$, so the membership condition
\revised{in the definition of $\mathcal{D}$ in (A3)} holds automatically
and the condition reduces to integrability of
$p_0\,e^{\inner{\lambda}{T}}$\revised{; differentiation under the
integral (A4) then holds on $\operatorname{ri}(\mathcal{D})$ by the
classical differentiability of log-partition-type integrals}. For
power/Tsallis generators, $I^{*}$ is a half-line, so the membership is a
genuine positivity restriction, and with continuous $\xi$ normalizability
can fail under heavy tails. \revised{(A3) and (A4) are hypotheses}
carried, not discharged for arbitrary $G$; \revised{they are} exactly
what \revised{make} the infinite-dimensional projection well-posed. For
the Shannon case with $G(t)=t\log t$, we have $G'(t)=1+\log t$ and
$\nabla F(u)=e^{u-1}$; hence $p_\lambda=p_0\,e^{\langle\lambda,T\rangle}$.
\end{remark}

\subsection{The dual function}

\begin{lemma}[Pointwise minimizer and dual function]
	\label{lem:pointwise}
	The Lagrangian
	\[
	L(p,\lambda)
	=
	D_G(p\,\|\,p_0)-\inner{\lambda}{\E_p[T]-c}
	\]
	is, for fixed $\lambda$, minimized over $p$ by
	\begin{equation}
		\label{eq:form}
		p^*_\lambda(\xi)
		=
		\nabla F\big(
		\nabla G(p_0(\xi))+\inner{\lambda}{T(\xi)}
		\big),
	\end{equation}
	whenever the argument lies in $I^*$ a.e. The corresponding concave dual
	function is
	\begin{equation}
		\label{eq:dual}
		{ 
			\Phi(\lambda)
			=
			\inner{\lambda}{c}
			-
			\int_\Xi
			F\big(\nabla G(p_0)+\inner{\lambda}{T}\big)\,d\mu
			+
			C_0,
			\qquad
			C_0=
			\int_\Xi F\big(\nabla G(p_0)\big)\,d\mu .
		}
	\end{equation}
	{Here $\Phi(\lambda)$ is finite only on the effective dual
		domain $\mathcal{D}$; outside $\mathcal{D}$ it may be read as $-\infty$.}
\end{lemma}

\begin{proof}
	The Lagrangian decouples across $\xi$ into strictly convex scalar
	problems. {At a fixed $\xi$, write $s=p(\xi)$ and evaluate
		$p_0$ and $T$ at that same point. Then
		\[
		\ell_\xi(s)
		=
		G(s)-G(p_0(\xi))
		-
		G'(p_0(\xi))(s-p_0(\xi))
		-
		\inner{\lambda}{T(\xi)}s.
		\]
		The first-order condition
		\[
		\ell_\xi'(s)
		=
		G'(s)-G'(p_0(\xi))-\inner{\lambda}{T(\xi)}
		=0
		\]
		gives \eqref{eq:form}. With
		\[
		u^*(\xi)
		=
		G'(p^*_\lambda(\xi))
		=
		\nabla G(p_0(\xi))+\inner{\lambda}{T(\xi)},
		\]
		Fenchel--Young gives
		\[
		G(p^*_\lambda(\xi))
		=
		u^*(\xi)p^*_\lambda(\xi)-F(u^*(\xi)).
		\]
		Therefore the coefficient of $p^*_\lambda(\xi)$ in
		$\ell_\xi(p^*_\lambda(\xi))$ vanishes, leaving
		\[
		\ell_\xi(p^*_\lambda(\xi))
		=
		F(\nabla G(p_0(\xi)))-F(u^*(\xi)).
		\]
		Integrating over $\xi$ and adding $\inner{\lambda}{c}$ yields
		\eqref{eq:dual}.}
\end{proof}

\begin{lemma}[Concavity, gradient, weak duality]
\label{lem:dual}
$\Phi$ is concave\revised{, with
$\nabla \Phi(\lambda)=c-\E_{p^{*}_\lambda}[T]$ on
$\operatorname{ri}(\mathcal{D})$, and} strictly concave \revised{on
$\mathcal{D}$} modulo
$\mathcal{N}=\{a\in\R^{k}:\inner{a}{T}=0\ \mu\text{-a.e.}\}$, and
$\Phi(\lambda)\le D_G(p\,\|\,p_0)$ for every $p\in\mathcal{M}(c)$ and
every $\lambda$.
\end{lemma}
\begin{proof}
Concavity follows because $F$ is convex and its argument is affine in
$\lambda$, so $-F(\cdots)$ is concave\revised{; with the convention
$\Phi=-\infty$ off $\mathcal{D}$, this gives concavity on all of
$\R^{k}$, $\mathcal{D}$ being convex}. Differentiating under the
integral, \revised{valid by (A4) on $\operatorname{ri}(\mathcal{D})$},
gives $\nabla_\lambda F(\cdots)=\nabla F(\cdots)\,T=p^{*}_\lambda T$,
hence $\nabla \Phi(\lambda)=c-\E_{p^{*}_\lambda}[T]$ \revised{there}.

\revised{Strictness off $\mathcal{N}$ is asserted on $\mathcal{D}$.
First, $\Phi$ is constant along $\mathcal{N}$: for $a\in\mathcal{N}$ the
integrand is unchanged, and $\inner{a}{c}=\E_p[\inner{a}{T}]=0$ for any
feasible $p$, which exists by (A2); hence $\Phi$ descends to
$\mathcal{D}/\mathcal{N}$. Now let
$\lambda_1\ne\lambda_2\in\mathcal{D}$ with
$\lambda_1-\lambda_2\notin\mathcal{N}$. The affine arguments
$\nabla G(p_0)+\inner{\lambda_1}{T}$ and
$\nabla G(p_0)+\inner{\lambda_2}{T}$ then differ on a set of positive
$\mu$-measure, on which $F$ --- strictly convex on $I^{*}$, since its
derivative $(G')^{-1}$ is strictly increasing --- satisfies the midpoint
inequality strictly; elsewhere it holds weakly. Integrating,
$\lambda\mapsto\int_\Xi F(\cdots)\,d\mu$ is strictly convex along the
segment $[\lambda_2,\lambda_1]\subset\mathcal{D}$, so $\Phi$ is strictly
concave on $\mathcal{D}$ modulo $\mathcal{N}$.}

For weak duality, \revised{if $\lambda\notin\mathcal{D}$ the bound is
trivial, since $\Phi(\lambda)=-\infty$; for $\lambda\in\mathcal{D}$,} if
$p$ is feasible, then $\E_p[T]-c=0$, so
$D_G(p\,\|\,p_0)=L(p,\lambda)\ge\inf_{p'}L(p',\lambda)=\Phi(\lambda)$,
\revised{the last equality by Lemma~\ref{lem:pointwise}}.
\end{proof}

\subsection{The Pythagorean theorem}

\begin{theorem}[e-projection: form, existence, uniqueness, Pythagoras]
	\label{thm:main}
	Under (A1)--(A4), problem \eqref{eq:primal} has a unique solution $p^*$
	$\mu$-a.e.,
	\begin{equation}
		\label{eq:soln}
		p^*(\xi)
		=
		\nabla F\big(
		\nabla G(p_0(\xi))+\inner{\lambda}{T(\xi)}
		\big),
	\end{equation}
	where $\lambda$ solves $\E_{p^*}[T]=c$ and is unique modulo
	$\mathcal{N}$. Thus the e-projection of $p_0$ onto the m-flat set
	{$\mathcal{M}(c)$} lands in
	{the $G$-exponential family
		$\mathcal{E}_G(p_0,T)$ of \S\ref{sec:prelim}; in the Shannon case this is
		the ordinary exponential family through $p_0$}. Moreover, for every
	$p\in{\mathcal{M}(c)}$ the \emph{Pythagorean identity} holds:
	\begin{equation}
		\label{eq:pyth}
		D_G(p\,\|\,p_0)
		=
		D_G(p\,\|\,p^*)+D_G(p^*\,\|\,p_0).
	\end{equation}
\end{theorem}

\begin{proof}
	\emph{Existence.} By Lemma~\ref{lem:dual}, {$\Phi$} is
	concave. By (A3), its supremum is attained at
	$\lambda^*\in\operatorname{ri}(\mathcal{D})$, where by (A4)
	{$\Phi$} is differentiable. The first-order condition gives
	\[
	{
		\nabla\Phi(\lambda^*)=0,
	}
	\]
	hence
	\[
	\E_{p^*_{\lambda^*}}[T]=c.
	\]
	The candidate $p^*:=p^*_{\lambda^*}$ is $I$-valued with finite divergence
	because $\lambda^*\in\mathcal{D}$, hence feasible.
	
	\emph{Optimality.} Since $p^*$ is feasible and minimizes
	$L(\cdot,\lambda^*)$ pointwise,
	\[
	D_G(p^*\,\|\,p_0)
	=
	{\Phi(\lambda^*)}
	\le
	D_G(p\,\|\,p_0)
	\]
	for all $p\in{\mathcal{M}(c)}$.
	
	\emph{Pythagoras and uniqueness.} For any feasible $p$, the three-point
	identity \eqref{eq:threepoint} with $(a,b,c)=(p,p_0,p^*)$, integrated over
	$\Xi$, has cross term
	\[
	\begin{aligned}
		\int_\Xi
		\big(G'(p^*)-G'(p_0)\big)(p-p^*)\,d\mu
		&=
		\int_\Xi \inner{\lambda}{T}(p-p^*)\,d\mu \\
		&=
		\inner{\lambda}{\E_p[T]-\E_{p^*}[T]}
		=
		0.
	\end{aligned}
	\]
	Here $G'(p^*)-G'(p_0)=\inner{\lambda}{T}$ by \eqref{eq:soln}, and both
	$p$ and $p^*$ meet the constraint. This gives \eqref{eq:pyth}. Therefore
	\[
	D_G(p\,\|\,p_0)\ge D_G(p^*\,\|\,p_0),
	\]
	with equality iff $D_G(p\,\|\,p^*)=0$. {Since the separable
		divergence is nonnegative and, by pointwise strict convexity of $G$,
		vanishes only where $p=p^*$, finiteness forces $p=p^*$ a.e. Uniqueness of
		$\lambda$ modulo $\mathcal{N}$ follows from strict concavity of $\Phi$ \revised{on
		$\mathcal{D}$ modulo $\mathcal{N}$ (Lemma~\ref{lem:dual})}.}
\end{proof}

\begin{remark}[Moment orthogonality]
	\label{rem:orth}
	The vanishing cross term in \eqref{eq:pyth} is the geometric heart of the
	theory. The residual direction
	\[
	\nabla G(p^*)-\nabla G(p_0)=\inner{\lambda}{T}
	\]
	lies in the span of the constraints. Hence {feasibility of
		both $p$ and $p^*$ makes the dual pairing
		\[
		\int_\Xi
		\big(\nabla G(p^*)-\nabla G(p_0)\big)(p-p^*)\,d\mu
		\]
		vanish: the natural-coordinate residual is orthogonal to the
		mean-coordinate displacement $p-p^*$.} Read as the dual stationary
	condition $\E_{p^*_\lambda}[T]=c$, the same equation is a $Z$-estimating
	equation (\S\ref{sec:dual}); read in the generalized linear model
	(\S\ref{sec:glm}) it is the score equation.
\end{remark}

\subsection{The m-projection, by conjugacy}

The Pythagorean theorem has a mirror image obtained for free. Projecting
in the \emph{second} argument onto an e-flat family
{$\mathcal{E}$},
\[
\min_{p\in\mathcal{E}}\ D_G(p_0\,\|\,p),
\]
is, by the dual-divergence identity \eqref{eq:dualdiv}, the e-projection
of $\nabla G(p_0)$ onto the m-flat image $\nabla G(\mathcal{E})$ with
respect to the conjugate divergence $D_F$
{under the analogues of (A1)--(A4) for the conjugate problem}.

\begin{corollary}[The two projections are one theorem]
	\label{cor:mproj}
	The m-projection
	\[
	\min_{p\in\mathcal{E}}D_G(p_0\,\|\,p)
	\]
	onto an e-flat family exists and is unique
	{under the corresponding assumptions for the conjugate
		generator $F$}, has the conjugate solution form, and obeys the dual
	Pythagorean identity
	\[
	D_F\big(\nabla G(p)\,\|\,\nabla G(p_0)\big)
	=
	D_F\big(\nabla G(p)\,\|\,\nabla G(p^*)\big)
	+
	D_F\big(\nabla G(p^*)\,\|\,\nabla G(p_0)\big),
	\]
	equivalently
	\[
	D_G(p_0\,\|\,p)
	=
	D_G(p^*\,\|\,p)
	+
	D_G(p_0\,\|\,p^*)
	\]
	in the m-coordinate, by applying Theorem~\ref{thm:main} to the generator
	$F$.
\end{corollary}

Thus the Pythagorean theorem produces both projections: the
\emph{e-projection} minimizes $D_G(\cdot\,\|\,p_0)$ onto an m-flat set,
the \emph{m-projection} minimizes $D_G(p_0\,\|\,\cdot)$ onto an e-flat set,
and the conjugacy $G\leftrightarrow F$ exchanges them. {No new
	proof is needed for the second direction; it is the same theorem applied
	in conjugate coordinates.}

\begin{remark}[``e-flat'' and ``m-flat'' are relative to the generator]
	\label{rem:relative-flat}
	Flatness is a property of a coordinate, so it swaps under conjugacy: a set
	affine in the e-coordinate $\theta=\nabla G(p)$ is e-flat in the
	$G$-geometry but m-flat in the $F$-geometry, where $\theta$ is the primal
	point. {Here $\theta$ is the primal variable for the
		$F$-geometry, while $p=\nabla F(\theta)$ is the primal variable for the
		$G$-geometry.} This is why one fit can carry two names. The generalized
	linear model of \S\ref{sec:glm} minimizes $D_F(\theta\,\|\,\theta_y)$ over
	a set affine in $\theta$: read in the natural coordinate, with $F$ primal,
	it is an e-projection onto an $F$-m-flat set; read in the mean coordinate,
	with $G$ primal, it is an m-projection onto a $G$-e-flat set. The label
	tracks which generator is taken as primal, not the fit itself.
\end{remark}

\section{Conjugate duality: the e/m swap}
\label{sec:dual}

The dual of the e-projection is the unconstrained concave program
\[
\max_{\lambda\in\mathcal{D}}\Phi(\lambda)
\]
of \eqref{eq:dual}, a generalized log-partition functional, or
equivalently an $F$-transform of the constraints {relative to
	the reference $p_0$}. Three propositions package the geometry of
\S\ref{sec:projections}.

\begin{proposition}[e/m swap under conjugacy]
	\label{prop:swap}
	The e-projection with respect to $D_G$ and the m-projection with respect
	to $D_F$ are one operation in conjugate coordinates. By
	\eqref{eq:dualdiv}, minimizing
	\[
	D_G(p\,\|\,p_0)
	\]
	over {$p\in\mathcal{M}(c)$} is equivalent to minimizing
	\[
	D_F\big(\nabla G(p_0)\,\|\,\nabla G(p)\big)
	\]
	over the image {$\nabla G(\mathcal{M}(c))$}. Thus the same
	problem is read in the conjugate coordinate system with the argument order
	reversed. The labels e- and m-projection are exchanged by
	$G\leftrightarrow F$, and the single Theorem~\ref{thm:main} covers both
	directions, {with Corollary~\ref{cor:mproj} giving the
		conjugate version under the corresponding assumptions.}
\end{proposition}

\begin{proposition}[Stationarity is an estimating equation]
	\label{prop:est}
	The dual stationary condition is
	\[
	\E_{p^*_\lambda}[T]=c,
	\]
	a $Z$-estimating equation in $\lambda$. If $G\in C^2(I)$ with $G''>0$ \revised{and differentiation under the
integral sign is valid to second order, then, for
$\lambda\in\operatorname{ri}(\mathcal{D})$,} 
	\[
	{
		-\nabla^2\Phi(\lambda)
		=
		\int_\Xi
		F''\big(\nabla G(p_0)+\inner{\lambda}{T}\big)\,
		T\,T^\top\,d\mu
		\succeq 0.
	}
	\]
	This matrix is the generalized information matrix of the estimating
	equation.
\end{proposition}

\begin{proposition}[The generalized evidence is generator-intrinsic]
	\label{prop:evid}
	Under (A1)--(A4), the optimal value
	\[
	{
		\mathcal{Z}_G
		:=
		\Phi(\lambda^*)
		=
		D_G(p^*\,\|\,p_0)
	}
	\]
	is well defined for every admissible generator $G$, independently of any
	relationship to the Shannon log-evidence. It is the $F$-transform of the
	constraints and the basic scalar output of the projection
	{--- an optimal dual value, not a normalizing constant. Hence
		the comparison with Shannon log-evidence is interpretive rather than
		literal.}
\end{proposition}

\begin{proof}
	{	Proposition~\ref{prop:swap} is the dual-divergence identity
		\eqref{eq:dualdiv}. \revised{For Proposition~\ref{prop:est}, the stationarity and the
gradient are Lemma~\ref{lem:dual}; differentiating
$\nabla\Phi(\lambda)=c-\int_\Xi
F'\big(\nabla G(p_0)+\inner{\lambda}{T}\big)\,T\,d\mu$
once more under the integral sign --- the hypothesis carried in the
statement --- gives the Hessian formula.}
        Proposition~\ref{prop:evid}
		follows from Theorem~\ref{thm:main} and the no-gap identity
		\[
		D_G(p^*\,\|\,p_0)=\Phi(\lambda^*).
		\]
	}
\end{proof}

With Part~\ref{part:math} complete, the geometry is fixed: a generator and
its conjugate, two coordinate systems, a projection theorem, a
Pythagorean identity, and the e/m pair it produces. The rest of the paper
interprets it.

\part{Statistical implications}
\label{part:stat}

Part~\ref{part:math} built one operation; Part~\ref{part:stat} reads the
statistics off it. Each method below fixes the three inputs of the
projection --- a generator, a reference point, and a constraint or model
family --- and instantiates the e-projection \eqref{eq:primal} or its
m-projection dual. {In the statistical sections the notation is
	specialized to each model: for example, the GLM link is denoted by
	$g_{\mathrm{link}}$ to avoid conflict with the dual objective
	$\Phi$ of Part~\ref{part:math}.} We begin with the generalized linear
model, where the two projections are clearest and most familiar: under
the canonical link, the score equation is exactly the Pythagorean
orthogonality, and the fit can be read in both natural and mean
coordinates. The remaining examples are arranged around this prototype.

\section{The generalized linear model is a Bregman projection}
\label{sec:glm}

{
	Throughout this section we distinguish between a general GLM and the
	canonical-link GLM. The Bregman interpretation exists for a general link
	as a divergence minimization over a possibly curved model manifold, but
	the projection theorem of Part~\ref{part:math} applies most directly in
	the canonical-link case, where the natural parameter is affine in the
	regression coefficients.
}

\subsection{The model is the geometry}
\label{sec:glm-model}

A generalized linear model is built from three ingredients
\cite{NelderWedderburn1972,McCullaghNelder1989}: a \emph{random component},
in which each $y_i$ follows an exponential-family law with cumulant $F$ (the log-partition function, written \revised{componentwise as
$b(\cdot)$} in the GLM literature); a
\emph{systematic component}, the linear predictor $X\beta$; and a
\emph{link} tying the mean to the predictor,
{  $g_{\mathrm{link}}(\mu)=X\beta$}. {  Each ingredient is a
	piece of the geometry of Part~\ref{part:math}. This subsection reads the
	model alone through that dictionary --- no data and no estimation yet, and
	the link left general.}

{ 
	The random component supplies the generator in the natural coordinate.}
For a regular exponential family, the cumulant function $F$ induces the
Bregman divergence attached to the response family through the
exponential-family--Bregman correspondence of \cite{Banerjee2005}. Its
Legendre dual
\[
G=F^*
\]
is the mean-domain generator and supplies the deviance below.

{The systematic component and the link together supply the model
	manifold.} Solving
\[
g_{\mathrm{link}}(\mu)=X\beta
\]
for the natural parameter gives, componentwise,
\[
{
	\mathcal{E}_X^{g}
	=
	\Big\{
	\theta(\beta)
	=
	(\nabla F)^{-1}\big(g_{\mathrm{link}}^{-1}(X\beta)\big):
	\beta\in\R^p
	\Big\}.
}
\]
This is a $p$-dimensional manifold in the natural coordinate. For a
general link it is curved. Under the canonical link,
\[
{
	g_{\mathrm{link}}=\nabla G=(\nabla F)^{-1},
}
\]
this manifold straightens to
\[
{
	\mathcal{E}_X
	=
	\{X\beta:\beta\in\R^p\}
	=
	\operatorname{col}(X),
}
\]
a linear subspace in the natural coordinate.

{	Thus a GLM supplies a generator and a model manifold. What it does not
	supply is the point to be projected; that point arrives with the data. The
	canonical link is the choice that makes the model manifold affine in the
	natural coordinate, and hence makes the projection theorem directly
	available.
}

\subsection{The likelihood is the divergence}

{	Estimation enters with the data, which supply the remaining ingredient of
	a projection: the point.} This point is the \emph{saturated} natural
parameter $\theta_y$, defined by
\[
\nabla F(\theta_y)=y,
\]
that is, the formal model point reproducing the observation exactly. With
natural parameters $\theta=(\theta_1,\dots,\theta_n)$ and cumulant
\[
F(\theta)=\sum_i b(\theta_i),
\]
the mean is
\[
\mu=\nabla F(\theta),
\]
and, up to a constant independent of $\theta$,
\[
\ell(\theta)=\inner{y}{\theta}-F(\theta)+\text{const}.
\] 
\revised{The dispersion parameter is set to one throughout; a known
common dispersion $\phi$ divides the objective by $\phi$ and rescales the scaled
deviance to $2D_G(y\,\|\,\hat\mu)/\phi$.} Using the Bregman convention of \eqref{eq:pointwise}, linearizing at the
second argument gives
\[
\begin{aligned}
	D_F(\theta\,\|\,\theta_y)
	&=
	F(\theta)-F(\theta_y)-\inner{\nabla F(\theta_y)}{\theta-\theta_y}  \\
	&=
	F(\theta)-\inner{y}{\theta}
	+
	\underbrace{\big[\inner{y}{\theta_y}-F(\theta_y)\big]}_{\text{const}} \\
	&=
	-\ell(\theta)+\text{const}.
\end{aligned}
\]
Thus, up to an additive constant independent of $\theta$, the negative
log-likelihood is the Bregman divergence from the fitted natural parameter
to the saturated natural parameter. Hence maximizing the likelihood is
equivalent to
\[
{
	\widehat\theta
	=
	\theta(\widehat\beta)
	=
	\arg\min_{\theta\in\mathcal{E}_X^{g}}
	D_F(\theta\,\|\,\theta_y).
}
\]
Under the canonical link this becomes
\[
	\widehat\theta
	=
	X\widehat\beta
	=
	\arg\min_{\theta\in\mathcal{E}_X}
	D_F(\theta\,\|\,\theta_y).
\]
{
	This is the natural-coordinate Bregman projection of the saturated model
	point onto the GLM model manifold. For a general link the target manifold
	$\mathcal{E}_X^g$ is curved; for the canonical link it is affine, so the
	Pythagorean theorem of Part~\ref{part:math} applies directly after reading
	the problem in the conjugate $F$-geometry.
}

\subsection{The score equation is the orthogonality}
\label{sec:glm-score}

For a general smooth strictly monotone link, stationarity of the
Bregman objective gives a weighted orthogonality relation. Since
\[
\nabla_\theta D_F(\theta\,\|\,\theta_y)
=
\nabla F(\theta)-\nabla F(\theta_y)
=
\mu-y,
\]
the chain rule along $\theta(\beta)$ gives
\[
\frac{\partial\theta_i}{\partial\beta}
=
\frac{x_i}{b''(\theta_i)\,g_{\mathrm{link}}'(\mu_i)}.
\]
Thus the optimum satisfies
\[
X^\top W(y-\hat\mu)=0,
\qquad
W=
\operatorname{diag}
\left\{
\frac{1}{b''(\theta_i)\,g_{\mathrm{link}}'(\mu_i)}
\right\}.
\]
The weights are parameter-dependent and sit between the residual and the
design.

{Among smooth monotone links, the canonical link is precisely the choice
	that removes these parameter-dependent weights and converts the estimating
	equation into ordinary orthogonality. Indeed,
	$W\equiv I$ exactly when
	\[
	g_{\mathrm{link}}'(\mu)b''(\theta)\equiv 1,
	\]
	equivalently when
	\[
	g_{\mathrm{link}}=\nabla G=(\nabla F)^{-1}
	\]
	up to an additive constant, which is absorbed by an intercept.}

The canonical link is therefore the Legendre dual map already determined
by the conjugacy $G\leftrightarrow F$. It is the link/inverse-link pair
\eqref{eq:links}, the same map that conjugates the divergence
(Lemma~\ref{lem:dualdiv}) and exchanges the dual e- and m-projections
(Proposition~\ref{prop:swap}). For the Gaussian, Bernoulli, and Poisson
cumulants it is the identity, the logit, and the log, respectively.

Under the canonical link, the model is
\[
\theta=X\beta,
\qquad
\mathcal{E}_X=\{X\beta:\beta\in\R^p\},
\]
and the MLE satisfies
\begin{equation}
	\label{eq:score}
	X^\top(y-\hat\mu)=0,
	\qquad
	\hat\mu=\nabla F(X\hat\beta).
\end{equation}
This is the score equation, obtained here as the stationarity of a
Bregman projection. The mean residual $y-\hat\mu$ is orthogonal to the
columns of $X$, i.e. to the tangent directions of the affine model
manifold. {It is the vanishing cross term of
	Remark~\ref{rem:orth} in concrete finite-dimensional form.}

\subsection{Both Pythagorean theorems}
\label{sec:glm-pyth}

With the canonical link in force, apply the three-point identity
\eqref{eq:threepoint} for the generator $F$ with
\[
(a,b,c)=(\theta,\theta_y,\hat\theta),
\qquad
\hat\theta=X\hat\beta.
\]
Then
\[
D_F(\theta\,\|\,\theta_y)
=
D_F(\theta\,\|\,\hat\theta)
+
D_F(\hat\theta\,\|\,\theta_y)
+
\inner{\hat\mu-y}{\theta-\hat\theta}.
\]
For any model point $\theta=X\beta$, the cross term is
\[
\inner{\hat\mu-y}{X(\beta-\hat\beta)}
=
-(\beta-\hat\beta)^\top X^\top(y-\hat\mu)
=
0
\]
by \eqref{eq:score}. Hence
\begin{equation}
	\label{eq:glm-prim}
	D_F(\theta\,\|\,\theta_y)
	=
	D_F(\theta\,\|\,\hat\theta)
	+
	D_F(\hat\theta\,\|\,\theta_y),
	\qquad
	\theta\in\mathcal{E}_X.
\end{equation}
This is the Pythagorean theorem in the natural coordinate, with $F$ taken
as the primal generator.

Translating each term through the dual-divergence identity
\eqref{eq:dualdiv},
\[
D_F(u\,\|\,v)
=
D_G(\nabla F(v)\,\|\,\nabla F(u)),
\]
with
\[
\mu=\nabla F(\theta),
\qquad
y=\nabla F(\theta_y),
\qquad
\hat\mu=\nabla F(\hat\theta),
\]
gives the mean-coordinate identity
\begin{equation}
	\label{eq:glm-dual}
	D_G(y\,\|\,\mu)
	=
	D_G(y\,\|\,\hat\mu)
	+
	D_G(\hat\mu\,\|\,\mu).
\end{equation}

{	Thus the same fit has two readings. In the natural coordinate it is the
	projection
	\[
	\min_{\theta\in\mathcal{E}_X}D_F(\theta\,\|\,\theta_y),
	\]
	an e-projection relative to the $F$-geometry. In the mean coordinate it is
	the dual m-projection
	\[
	\min_{\mu\in\nabla F(\mathcal{E}_X)}D_G(y\,\|\,\mu),
	\]
	where the second argument varies over the mean-manifold. The labels
	``e'' and ``m'' are therefore relative to the chosen primal generator, as
	in Remark~\ref{rem:relative-flat}.} The model deviance is
\[
D_G(y\,\|\,\hat\mu),
\]
up to the conventional factor of two; the textbook deviance is
$2D_G(y\,\|\,\hat\mu)$.

\begin{figure}[ht]
	\centering
	\begin{tikzpicture}[>={Stealth[length=2.2mm]}, thick,
		dot/.style={circle,fill,inner sep=1.2pt}]
		
		\draw[rounded corners, draw=blue!45, fill=blue!6] (0,0) rectangle (5.2,4.8);
		\node[align=center, blue!55!black] at (2.6,4.35)
		{\footnotesize\textbf{Natural coordinate} $\theta$ \;($e$-space)};
		\draw[thick] (0.7,1.35) -- (4.6,1.35);
		\node[above] at (3.75,1.4) {\footnotesize $\mathcal{E}_X=\{X\beta\}$ ($e$-flat)};
		\node[dot] (ty) at (2.2,3.45) {};
		\node[above] at (2.2,3.5) {\footnotesize $\theta_y$ (saturated)};
		\node[dot] (th) at (2.2,1.35) {};
		\draw (ty) -- (th);
		\draw (2.2,1.6) -- (2.45,1.6) -- (2.45,1.35);
		\node[below] at (2.12,1.3) {\footnotesize $\hat\theta$};
		\node[left] at (2.1,2.45) {\footnotesize $e$-projection};
		
		\draw[rounded corners, draw=red!45, fill=red!6] (7.4,0) rectangle (12.6,4.8);
		\node[align=center, red!55!black] at (10.0,4.35)
		{\footnotesize\textbf{Mean coordinate} $\mu$ \;($m$-space)};
		\draw[thick] (8.0,1.5) to[out=22,in=158] (12.0,1.5);
		\node[below] at (11.0,1.55) {\footnotesize $\nabla F(\mathcal{E}_X)$ (curved)};
		\node[dot] (yy) at (9.5,3.45) {};
		\node[above] at (9.5,3.5) {\footnotesize $y$ (data)};
		\node[dot] (mh) at (9.5,2.12) {};
		\draw (yy) -- (mh);
		\node[below left] at (9.58,2.12) {\footnotesize $\hat\mu$};
		\node[right, align=left] at (9.62,2.75)
		{\footnotesize $m$-projection};
		\node[right] at (9.62,2.45) {\footnotesize deviance $D_G(y\,\|\,\hat\mu)$};
		
		\draw[->] (5.35,2.95) -- (7.25,2.95);
		\node[above, align=center] at (6.3,3.0)
		{\footnotesize $\nabla F$\\[-2pt]\scriptsize (inverse link)};
		\draw[->] (7.25,1.55) -- (5.35,1.55);
		\node[below, align=center] at (6.3,1.5)
		{\footnotesize $\nabla G$\\[-2pt]\scriptsize (link)};
	\end{tikzpicture}
	\caption{The dual map underlying a canonical-link generalized linear model.
		The same fit is read in two coordinate systems linked by the Legendre maps
		$\nabla F$ and $\nabla G$. In the natural coordinate $\theta$, the model
		$\mathcal{E}_X=\{X\beta\}$ is affine, and the fit is the projection
		$\min_{\theta\in\mathcal{E}_X}D_F(\theta\,\|\,\theta_y)$ of the saturated
		model point $\theta_y$ onto it. The score equation \eqref{eq:score} is the
		orthogonality condition, and \eqref{eq:glm-prim} is the corresponding
		Pythagorean identity. Applying $\nabla F$ carries the picture to the mean
		coordinate $\mu$, where $\theta_y$ becomes the data $y$, the same fit is
		the dual m-projection
		$\min_{\mu\in\nabla F(\mathcal{E}_X)}D_G(y\,\|\,\mu)$, and
		\eqref{eq:glm-dual} measures the deviance through
		$D_G(y\,\|\,\hat\mu)$, up to the conventional factor of two.
		\revised{The right angle is schematic: the orthogonality is the
		dual pairing of Remark~\ref{rem:orth}, between the mean residual
		$y-\hat\mu$ and the columns of $X$, and reduces to a Euclidean
		angle in $\theta$ only in the Gaussian case.}        
        }
	\label{fig:glm-dual}
\end{figure}
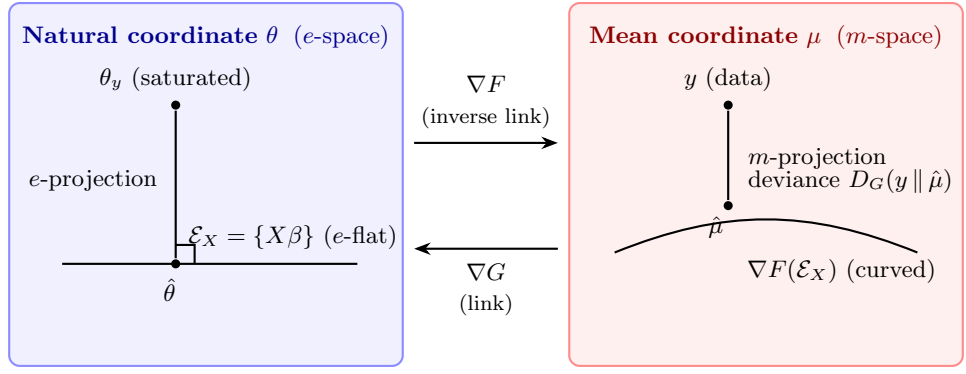

\subsection{Three worked cases}

\begin{example}[Gaussian / least squares]
	$b(\theta)=\theta^2/2$, so $\mu=\theta$ and
	\[
	D_G(y\,\|\,\mu)=\frac12\|y-\mu\|^2.
	\]
	The projection is the ordinary orthogonal projection of $y$ onto
	$\operatorname{col}(X)$. The score equation \eqref{eq:score} is the
	normal equation
	\[
	X^\top(y-X\hat\beta)=0,
	\]
	and \eqref{eq:glm-prim} is the least-squares Pythagorean decomposition.
\end{example}

\begin{example}[Bernoulli / logistic regression]
	$b(\theta)=\log(1+e^\theta)$, so
	\[
	\mu=(1+e^{-\theta})^{-1}.
	\]
	The mean-domain divergence is
	\[
	D_G(y\,\|\,\mu)
	=
	\sum_i
	\left[
	y_i\log\frac{y_i}{\mu_i}
	+
	(1-y_i)\log\frac{1-y_i}{1-\mu_i}
	\right],
	\]
	with the usual conventions $0\log0=0$ and $0\log(0/\mu)=0$. Thus logistic
	regression is the Bernoulli-KL projection of the responses onto the
	logistic mean manifold, provided the finite MLE exists; see
	Remark~\ref{rem:glm-boundary}.
\end{example}

\begin{example}[Poisson regression]
	$b(\theta)=e^\theta$, so $\mu=e^\theta$, and
	\[
	D_G(y\,\|\,\mu)
	=
	\sum_i
	\left[
	y_i\log\frac{y_i}{\mu_i}-y_i+\mu_i
	\right].
	\]
	The conventional Poisson deviance is $2D_G(y\,\|\,\mu)$, minimized by the
	projection.
\end{example}

\begin{remark}[Boundary cases and existence]
	\label{rem:glm-boundary}
	For Bernoulli and Poisson responses, the saturated natural parameter can
	be infinite: $\theta_{y_i}=\pm\infty$ when $y_i\in\{0,1\}$ for the logit,
	and $\theta_{y_i}=-\infty$ when $y_i=0$ for Poisson. Thus
	$\nabla F(\theta_y)=y$ and $F(\theta_y)$ are read in the extended closure
	sense, and the deviance is defined by the corresponding limiting values.
	Existence of the fit is likewise not automatic: the projection is
	attained only when $\hat\theta=X\hat\beta$ is finite. This can fail under
	complete or quasi-complete separation in logistic regression. {
		Finite existence can also fail for Poisson regression in boundary cases,
		depending on the design and zero-count patterns; logistic separation is
		simply the most familiar instance.}
\end{remark}

{For canonical-link regular exponential families, maximum likelihood can
	therefore be interpreted as a direct instance of the projection theorem of
	Part~\ref{part:math}: the cumulant supplies the natural-coordinate
	generator, the canonical link supplies the Legendre dual map, the design
	matrix supplies the affine model manifold, and the saturated model supplies
	the point to be projected. The score equation is the Pythagorean
	orthogonality, the two coordinate systems give the two Pythagorean
	decompositions, and the deviance is twice the mean-coordinate Bregman
	divergence minimized by the fit.
}

\section{Maximum entropy and maximum likelihood}
\label{sec:maxent}

Work in the space of distributions on $\Xi$ with the Shannon generator
\[
G(t)=t\log t,
\]
reference distribution $p_0$, and statistic $T:\Xi\to\R^k$. Two flat
families meet here. The first is the moment set
\[
\mathcal{M}(c)=\{p:\E_p[T]=c\},
\]
which is m-flat. The second is the exponential family
\[
{
	\mathcal{E}
	=
	\left\{
	p_\lambda
	=
	\frac{p_0 e^{\inner{\lambda}{T}}}{Z(\lambda)}
	:
	\lambda\in\R^k,\ 
	Z(\lambda)=\int_\Xi p_0 e^{\inner{\lambda}{T}}\,d\mu<\infty
	\right\},
}
\]
which is e-flat in the Shannon geometry. {The normalizing
	constant is included here because, in this section, we work with
	probability distributions rather than arbitrary positive measures.}

\paragraph{Maximum entropy is the e-projection.}
Minimizing
\[
D_G(p\,\|\,p_0)=\KL(p\,\|\,p_0)
\]
over $\mathcal{M}(c)$ is the e-projection of $p_0$ onto the moment set $\mathcal{M}(c)$.
By Theorem~\ref{thm:main}, its solution has the Gibbs form
\[
p^*
=
p_{\lambda^*}
=
\frac{p_0 e^{\inner{\lambda^*}{T}}}{Z(\lambda^*)}
\in\mathcal{E},
\]
where $\lambda^*$ is chosen so that
\[
\E_{p^*}[T]=c.
\]
{	This statement presumes feasibility, attainment, and finiteness of
	$Z(\lambda)$ in a neighborhood of $\lambda^*$; these are the Shannon
	versions of the standing assumptions in \S\ref{sec:projections}.}
Thus the e-projection generates the exponential family: a prescribed mean
on the positive half-line gives an exponential law, prescribed mean and
variance on $\R$ give a Gaussian, and so on.

\paragraph{Maximum likelihood is the m-projection.}
Given data $x_1,\dots,x_n$, let
\[
\hat p=\frac1n\sum_{i=1}^n\delta_{x_i},
\qquad
\hat c=\E_{\hat p}[T]=\frac1n\sum_{i=1}^nT(x_i).
\]
Fitting the exponential family by maximum likelihood is equivalent to
minimizing
\[
D_G(\hat p\,\|\,p_\lambda)
=
\KL(\hat p\,\|\,p_\lambda)
=
-H(\hat p)-\E_{\hat p}[\log p_\lambda]
\]
over $\lambda$, because $H(\hat p)$ is constant in $\lambda$ and
$-\E_{\hat p}[\log p_\lambda]$ is the negative average log-likelihood.
For a discrete sample space this is literally a KL divergence. For
continuous models, $\hat p$ is discrete while $p_\lambda$ is usually a
density with respect to $\mu$; hence the expression is read as the
empirical cross-entropy
\[
-\frac1n\sum_{i=1}^n\log p_\lambda(x_i),
\]
rather than as a literal density KL. {Equivalently, the
	continuous case is the usual likelihood calculation written in the
	formal KL notation.} This is the m-projection of the empirical
distribution $\hat p$ onto the e-flat family $\mathcal{E}$, varying the
second argument of the divergence.

\begin{proposition}[Maximum-entropy / maximum-likelihood duality]
	\label{prop:maxent-ml}
	\revised{Let the maximum-entropy reference and the base point of
	$\mathcal{E}$ be the same $p_0$, as above, and take the
	maximum-entropy target to be the empirical moment $\hat c$; the
	coincidence asserted below depends on both identifications.}
	Assume the exponential family is regular and minimal modulo the null space
	$\mathcal{N}$, the empirical moment $\hat c$ lies in the relative interior
	of the attainable moment set, and the log-partition function is finite in
	a neighborhood of the optimum. Then the e-projection of $p_0$ onto
	$\mathcal{M}(\hat c)$ and the m-projection of $\hat p$ onto
	$\mathcal{E}$ coincide. Their common value is the unique distribution
	\[
	p_{\hat\lambda}\in \mathcal{E}\cap\mathcal{M}(\hat c).
	\]
	Moreover, for every $p\in\mathcal{M}(\hat c)$ and every
	$q\in\mathcal{E}$,
	\begin{equation}
		\label{eq:maxent-ml-pyth}
		D_G(p\,\|\,p_0)
		=
		D_G(p\,\|\,p_{\hat\lambda})
		+
		D_G(p_{\hat\lambda}\,\|\,p_0),
		\qquad
		D_G(\hat p\,\|\,q)
		=
		D_G(\hat p\,\|\,p_{\hat\lambda})
		+
		D_G(p_{\hat\lambda}\,\|\,q).
	\end{equation}
\end{proposition}

\revised{In the continuous case the terms $D_G(\hat p\,\|\,q)$ and
$D_G(\hat p\,\|\,p_{\hat\lambda})$ in the second identity are read as
empirical cross-entropies, as above; the identity holds in that reading
because the formally divergent entropy of $\hat p$ is common to both
sides and cancels.}

\begin{proof}
	For $p_\lambda\in\mathcal{E}$,
	\[
	\log p_\lambda
	=
	\log p_0+\inner{\lambda}{T}-\log Z(\lambda).
	\]
	Thus the average log-likelihood is, up to terms independent of $\lambda$,
	\[
	\E_{\hat p}[\log p_\lambda]
	=
	\inner{\lambda}{\hat c}-\log Z(\lambda)+\text{const}.
	\]
	Its stationarity condition is
	\[
	\nabla_\lambda
	\bigl[
	\inner{\lambda}{\hat c}-\log Z(\lambda)
	\bigr]
	=
	\hat c-\E_{p_\lambda}[T]
	=
	0.
	\]
	Therefore the MLE $p_{\hat\lambda}$ matches the empirical moments:
	\[
	\E_{p_{\hat\lambda}}[T]=\hat c.
	\]
	Hence
	\[
	p_{\hat\lambda}\in\mathcal{E}\cap\mathcal{M}(\hat c).
	\]
	
	On the other hand, the maximum-entropy solution with target
	$c=\hat c$ has the Gibbs form
	\[
	p^*
	=
	\frac{p_0 e^{\inner{\lambda^*}{T}}}{Z(\lambda^*)}
	\in\mathcal{E}
	\]
	and satisfies
	\[
	\E_{p^*}[T]=\hat c.
	\]
	Thus it also lies in $\mathcal{E}\cap\mathcal{M}(\hat c)$. Under the
	regularity and minimality assumptions, the map
	$\lambda\mapsto\E_{p_\lambda}[T]$ is injective modulo $\mathcal{N}$, so
	the intersection contains a unique finite point. Hence the maximum-entropy
	and maximum-likelihood solutions coincide.
	
	The first identity in \eqref{eq:maxent-ml-pyth} is the Pythagorean
	identity of the e-projection onto the moment set. For the second identity,
	apply the three-point identity \eqref{eq:threepoint} with
	\[
	(a,b,c)=(\hat p,q,p_{\hat\lambda}).
	\]
	Since both $q$ and $p_{\hat\lambda}$ lie in the exponential family,
	\[
	\nabla G(p_{\hat\lambda})-\nabla G(q)
	\]
	lies in the span of the sufficient statistics \revised{and the
	constant functions, the constant being the difference of the
	log-normalizers}. Since both $\hat p$ \revised{(by the definition
	$\hat c=\E_{\hat p}[T]$)} and
	$p_{\hat\lambda}$ have moment $\hat c$\revised{, and both have unit
	total mass, which disposes of the constant term}, the integrated cross term
	vanishes. This gives the second Pythagorean identity. {Thus
		the two decompositions are the e- and m-sides of the same intersection
		geometry.}
\end{proof}

\begin{remark}[Existence of the intersection]
	\label{rem:maxent-exist}
	The unique finite intersection point presumes the usual regularity: the
	exponential family is minimal modulo $\mathcal{N}$, the target moment
	$\hat c$ lies in the relative interior of the attainable moment set, and
	the log-partition function is finite near $\hat\lambda$. If $\hat c$ lies
	on the boundary, the statement must be read in the closure, and the MLE
	may fail to exist as a finite point. This is the same boundary phenomenon
	as in Remark~\ref{rem:glm-boundary}.
\end{remark}

Maximum entropy and maximum likelihood are therefore not merely the same
procedure written in different coordinates. They are two different
projections: maximum entropy projects the reference $p_0$ onto the moment
foliation, while maximum likelihood projects the empirical distribution
$\hat p$ onto the exponential foliation. They arrive at the same point
because the m-flat and e-flat families intersect there. The exponential
family is at once the image of the e-projection and the target of the
m-projection, and the two Pythagorean identities in
\eqref{eq:maxent-ml-pyth} record the two sides of that single
intersection. {Figure~\ref{fig:maxent-ml} depicts this
	intersection geometry.}

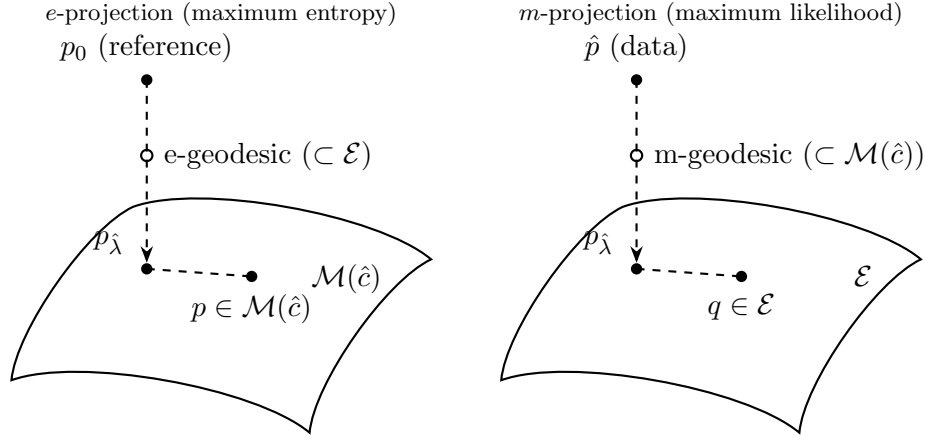
\begin{figure}[ht]
	\centering
	\begin{tikzpicture}[x={(170:1cm)},y={(55:.7cm)},z={(90:1cm)},
		>={Stealth[length=2.2mm]}, thick,
		dot/.style={circle,fill,inner sep=1.5pt}]
		\node at (0,0,3.05) {\footnotesize $e$-projection (maximum entropy)};
		\draw[solid,looseness=.6] (2.0,-2.0,-1)
		to[bend left] (2.0,2.0,-1)
		to[bend left] (-2.0,2.0,-1)
		to[bend right] (-2.0,-2.0,-1)
		to[bend right] (2.0,-2.0,-1)
		-- cycle;
		\node at (-1.15,1.3,-1.0) {$\mathcal{M}(\hat c)$};
		\node[dot, label={above:$p_0$ (reference)}] at (1,0,2) {};
		\node[dot, label={[label distance=2.5pt]162:$p_{\hat\lambda}$}] at (1,0,-0.5) {};
		\node[fill=white, draw, shape=circle, inner sep=1.5pt,
		label={right:e-geodesic ($\subset\mathcal{E}$)}] at (1,0,1) {};
		\draw[dashed,->] (1,0,2) -- (1,0,-0.42);
		\node[dot, label={below:$p\in\mathcal{M}(\hat c)$}] at (0,1,-1.0) {};
		\draw[dashed] (1,0,-0.5) -- (0,1,-1.0);
	\end{tikzpicture}
	\qquad
	\begin{tikzpicture}[x={(170:1cm)},y={(55:.7cm)},z={(90:1cm)},
		>={Stealth[length=2.2mm]}, thick,
		dot/.style={circle,fill,inner sep=1.5pt}]
		\node at (0,0,3.05) {\footnotesize $m$-projection (maximum likelihood)};
		\draw[solid,looseness=.6] (2.0,-2.0,-1)
		to[bend left] (2.0,2.0,-1)
		to[bend left] (-2.0,2.0,-1)
		to[bend right] (-2.0,-2.0,-1)
		to[bend right] (2.0,-2.0,-1)
		-- cycle;
		\node at (-1.45,1.45,-1.0) {$\mathcal{E}$};
		\node[dot, label={above:$\hat p$ (data)}] at (1,0,2) {};
		\node[dot, label={[label distance=2.5pt]162:$p_{\hat\lambda}$}] at (1,0,-0.5) {};
		\node[fill=white, draw, shape=circle, inner sep=1.5pt,
		label={right:m-geodesic ($\subset\mathcal{M}(\hat c)$)}] at (1,0,1) {};
		\draw[dashed,->] (1,0,2) -- (1,0,-0.42);
		\node[dot, label={below:$q\in\mathcal{E}$}] at (0,1,-1.0) {};
		\draw[dashed] (1,0,-0.5) -- (0,1,-1.0);
	\end{tikzpicture}
	\caption{The two projections of Proposition~\ref{prop:maxent-ml}. Left:
		maximum entropy is the e-projection of the reference $p_0$ onto the
		moment set $\mathcal{M}(\hat c)$; the projection path is an e-geodesic
		inside the exponential family $\mathcal{E}$. Right: maximum likelihood is
		the m-projection of the empirical distribution $\hat p$ onto
		$\mathcal{E}$; the projection path is the m-geodesic
		$(1-t)\hat p+t\,p_{\hat\lambda}$, which remains inside
		$\mathcal{M}(\hat c)$ because the moments stay equal to $\hat c$. The two
		foot points coincide at the unique intersection
		$\mathcal{E}\cap\mathcal{M}(\hat c)=\{p_{\hat\lambda}\}$.}
	\label{fig:maxent-ml}
\end{figure}

	\paragraph{Beyond the Shannon generator.}
	Nothing in the geometry above is special to Shannon. Theorem~\ref{thm:main}
	was proved for an arbitrary admissible generator. For such a generator,
	the e-projection of $p_0$ onto $\mathcal{M}(\hat c)$ has the solution form
	\[
	p_\lambda
	=
	\nabla F\big(\nabla G(p_0)+\inner{\lambda}{T}\big),
	\]
	which traces out a deformed exponential family
	\[
	\mathcal{E}_G
	=
	\{p_\lambda:\lambda\in\mathcal{D}\}.
	\]
	This is the U-model of \cite{Eguchi}; it reduces to the Gibbs family only
	when $G(t)=t\log t$.
	
	What is special to Shannon is the classical likelihood interpretation of
	the m-side. For a general separable Bregman divergence,
	\[
	\begin{aligned}
		D_G(\hat p\,\|\,q)
		&=
		\int_\Xi G(\hat p)\,d\mu
		-
		\int_\Xi G(q)\,d\mu
		-
		\int_\Xi G'(q)(\hat p-q)\,d\mu  \\
		&=
		\underbrace{\int_\Xi G(\hat p)\,d\mu}_{\text{constant in }q}
		+
		\int_\Xi\bigl[G'(q)q-G(q)\bigr]\,d\mu
		-
		\E_{\hat p}[G'(q)].
	\end{aligned}
	\]
	The only term coupling the data and the model is
	$\E_{\hat p}[G'(q)]$, which is affine in $\hat p$ and can be estimated by
	the sample average
	\[
	\frac1n\sum_{i=1}^n G'(q(x_i)).
	\]
	Thus minimizing $D_G(\hat p\,\|\,q)$ over a model family is generally an
	M-estimation problem associated with a strictly proper scoring rule
	\cite{GneitingRaftery2007,GrunwaldDawid2004}. The Shannon generator gives
	the logarithmic score and recovers maximum likelihood. Power generators
	lead to density-power divergences and robust minimum-divergence
	estimators \cite{Basu1998}.
	
	The duality still survives at the projection level. Along the deformed
	family $\mathcal{E}_G$, the e-coordinate $\nabla G(p_\lambda)$ is affine
	in $T$ by construction, so the data enter the stationarity equation only
	through the empirical moment $\hat c$. The stationary condition is again
	moment matching:
	\[
	\E_{p_\lambda}[T]=\hat c.
	\]
	Consequently generalized maximum entropy and minimum-score estimation meet
	at the intersection
	\[
	\mathcal{E}_G\cap\mathcal{M}(\hat c),
	\]
	when that intersection exists and is unique.
	
	What breaks outside Shannon is the classical exponential-family
	superstructure. The deformed family $\mathcal{E}_G$ need not have a
	product-form likelihood or a finite-dimensional sufficient statistic in
	the usual sense; its dual objective is the integral functional
	$\Phi(\lambda)$ of Theorem~\ref{thm:main}, not necessarily a classical
	log-partition function. Positivity and normalizability are also no longer
	automatic, since $\nabla F$ may reach the boundary of $I$. Thus the choice
	of generator is a statistical choice: it selects a deformed family, a
	proper scoring rule, and an estimator. Section~\ref{sec:future} returns to
	this generator-selection problem.

\section{Survey calibration and the over-identified model}
\label{sec:calib}

{ 
	Consider a finite population $\mathcal{U}=\{1,\dots,N\}$ carrying an
	auxiliary vector $x_i\in\R^k$ and a study variable $y_i$. A sample
	$A\subset\mathcal{U}$ of size $n$ is selected according to a known
	probability sampling design. Thus the first-order inclusion probabilities
	$\pi_i=\Pr(i\in A)>0$ 
	are available for the sampled units, and the design weights
	$d_i=\pi_i^{-1}$ yield the Horvitz--Thompson estimator
	\[
	\hat Y_{\mathrm{HT}}=\sum_{i\in A}d_i y_i
	\]
	for the population total $Y=\sum_{i\in\mathcal{U}}y_i$
	\cite{HorvitzThompson1952}. Often the auxiliary total
	\[
	X=\sum_{i\in\mathcal{U}}x_i
	\]
	is known exactly, while the weighted sample analogue
	$\sum_{i\in A}d_i x_i$ differs from $X$ because of sampling error.
	Calibration \cite{DevilleSarndal1992} adjusts the design weights as little
	as possible while forcing the calibrated weights to reproduce the known
	benchmark total. In the notation of Part~\ref{part:math}, the sample
	index set $A$ plays the role of $\Xi$, the weight vector
	$w=(w_i)_{i\in A}$ plays the role of a positive, unnormalized density, and
	the benchmark equation cuts out an affine moment set.
}

\subsection{Calibration as a single e-projection}

{ 
	The Bregman calibration problem is
}
\[
\min_{\revised{w\in I^{A}}}\ D_G(w\,\|\,d)
\qquad
\text{subject to}
\qquad
\sum_{i\in A}w_i x_i=X,
\]
{ 
	where
}
\[
D_G(w\,\|\,d)
=
\sum_{i\in A}
\bigl[
G(w_i)-G(d_i)-G'(d_i)(w_i-d_i)
\bigr].
\]
{ 
	Thus calibration is the e-projection of the design-weight vector $d$ onto
	the affine moment set
}
\[
\mathcal{M}(X)
=
\left\{
w\in \revised{I^{A}}:
\sum_{i\in A}w_i x_i=X
\right\}.
\]
{ 
	No simplex constraint is imposed: the weights are positive measures, not
	probability vectors.}
\revised{Positivity is a property of the generator rather than an
imposed constraint: for generators with $I=(0,\infty)$, such as Shannon,
the calibrated weights are automatically positive, whereas the quadratic
generator has $I=\R$ and can return negative weights --- the familiar
drawback of regression-type calibration. The contrast-entropy generator
below has $I=(1,\infty)$ and keeps every calibrated weight above one.}

By Theorem~\ref{thm:main}, the calibrated weights have the form
\begin{equation}
	\label{eq:calib}
	w_i
	=
	\nabla F\bigl(\nabla G(d_i)+\lambda^\top x_i\bigr),
\end{equation}
where $\lambda$ is determined by the calibration equation
\[
\sum_{i\in A}w_i(\lambda)x_i=X.
\]
{ 
	The formula is valid whenever
	$\nabla G(d_i)+\lambda^\top x_i$ lies in the dual domain $I^*$\revised{,
	in which case the weights automatically lie in $I$}.}

For the Shannon generator $G(t)=t\log t$, one obtains the raking form
\[
w_i=d_i e^{\lambda^\top x_i}.
\]
For the quadratic generator $G(t)=t^2/2$, one obtains the additive form
\[
w_i=d_i+\lambda^\top x_i.
\]
{ 
	The familiar multiplicative GREG form
	$w_i=d_i(1+\lambda^\top x_i)$ is different: it arises from a chi-square
	distance on the ratios $w_i/d_i$, namely
	$\sum_i d_i\,\phi(w_i/d_i)$ with
	$\phi(r)=\tfrac12(r-1)^2$, rather than from the unweighted separable
	Bregman divergence $D_G(w\,\|\,d)$.}

Thus calibration is the same e-projection mechanism as maximum entropy,
but with a non-uniform reference $d$ and with unnormalized positive
weights. Its dual equation
\[
\sum_{i\in A}w_i(\lambda)x_i=X
\]
is formally analogous to the score equation \eqref{eq:score}. There are
$k$ benchmark equations and $k$ Lagrange multipliers, so this is the
just-identified case. {  If the benchmarks are linearly
	dependent, the multiplier is unique only modulo the null space
	$\mathcal{N}$, exactly as in Theorem~\ref{thm:main}.}

{ 
	\paragraph{The projection is asymptotically a regression estimator.}
	The calibrated weights are used to estimate the study total by
}
\[
\hat Y_w=\sum_{i\in A}w_i y_i.
\]
{ 
	Kim et al.  \cite{kim2026} shows that, under design-based
	regularity conditions, this estimator is asymptotically a regression
	estimator:
}
\[
\hat Y_w
=
\hat Y_{\mathrm{HT}}
+
\bigl(X-\hat X_{\mathrm{HT}}\bigr)^\top\beta_G
+
o_p\bigl(N/\sqrt n\bigr),
\]
where
\[
\hat X_{\mathrm{HT}}=\sum_{i\in A}d_i x_i,
\qquad
\beta_G
=
\left(
\sum_{i\in\mathcal{U}}\pi_i q_i x_i x_i^\top
\right)^{-1}
\sum_{i\in\mathcal{U}}\pi_i q_i x_i y_i,
\qquad
q_i=\frac{1}{G''(d_i)}.
\]
{ 
	This expression has the usual regression-estimator interpretation: it is
	the Horvitz--Thompson estimator corrected by the discrepancy between the
	known auxiliary total $X$ and its Horvitz--Thompson estimate
	$\hat X_{\mathrm{HT}}$. The correction is design-consistent whether or not
	the linear working model is true.}

The dependence on $G$ is forced by the local geometry. Differentiating
\eqref{eq:calib} at $\lambda=0$ gives
\[
w_i
\approx
d_i
+
\frac{1}{G''(d_i)}\lambda^\top x_i
=
d_i+q_i\lambda^\top x_i.
\]
Therefore, near the design weights, every separable Bregman calibration is
an additive adjustment whose curvature weights are
$q_i=1/G''(d_i)$. These weights propagate into the regression coefficient
$\beta_G$.

{ 
	This is where the separable Bregman formulation differs from classical
	Deville--S\"arndal calibration. In the classical ratio-distance form, the
	distance is applied to $w_i/d_i$, so the first-order GREG coefficient is
	the same for a broad class of distances. In the separable formulation, the
	curvature of $G$ remains visible at first order through $q_i$, so the
	choice of generator is a genuine efficiency choice rather than a purely
	notational choice.}

\subsection{The over-identified model}

Calibration imposes one fixed slice of affine constraints. A moment model
imposes a parametric family of such slices and asks which slice is closest
to the empirical distribution.

Let $Z_1,\dots,Z_n$ be i.i.d.\ observations and let
\[
U:\Theta\times\mathcal{Z}\to\R^r
\]
be an estimating function. Suppose the true parameter $\theta_0$ is
identified by
\[
\E_{P_0}[U(\theta_0;Z)]=0,
\qquad
\Theta\subset\R^p.
\]
When $r>p$, the model is over-identified: {  the parameter
	$\theta$ has only $p$ degrees of freedom, whereas the moment equation has
	$r$ components, so the sample moments generally cannot all be set to zero
	by choosing $\theta$.} The remaining $r-p$ restrictions are the
overidentifying restrictions \cite{Hansen1982}.

Work on positive measures on the sample:
\[
\omega=(\omega_1,\dots,\omega_n)\in\revised{I^n},
\]
anchored at the empirical distribution
\[
\hat P_n=(1/n,\dots,1/n).
\]
{ 
	Again, no simplex constraint is imposed; this is the positive-measure
	analogue of the normalized empirical-likelihood and exponential-tilting
	formulations.}

For fixed $\theta$, define the affine moment slice
\[
\mathcal{M}(\theta)
=
\left\{
\omega\in \revised{I^n}:
\sum_{i=1}^n\omega_i U(\theta;Z_i)=0
\right\}.
\]
The slice projection is
\[
\hat\omega(\theta)
=
\arg\min_{\omega\in\mathcal{M}(\theta)}
D_G(\omega\,\|\,\hat P_n).
\]
By Theorem~\ref{thm:main}, whenever this projection exists, it has the
form
\[
\hat\omega_i(\theta)
=
\nabla F
\bigl(
G'(1/n)+\hat\lambda_\theta^\top U(\theta;Z_i)
\bigr),
\]
where $\hat\lambda_\theta$ solves
\[
\sum_{i=1}^n
\hat\omega_i(\theta)U(\theta;Z_i)
=
0.
\]
{ 
	Thus each fixed-$\theta$ step is exactly a calibration problem with target
	total zero and auxiliary vector $U(\theta;Z_i)$.}

The over-identified estimator then chooses the slice with the smallest
projection cost:
\begin{equation}
	\label{eq:overid-est}
	\hat\theta_G
	=
	\arg\min_{\theta\in\Theta}
	D_G\bigl(\hat\omega(\theta)\,\|\,\hat P_n\bigr),
	\qquad
	\revised{\hat\omega(\hat\theta_G)
	=
	\arg\min_{\omega\in\mathcal{M}_\Theta}
	D_G(\omega\,\|\,\hat P_n)},
	\qquad
	\mathcal{M}_\Theta
	=
	\bigcup_{\theta\in\Theta}\mathcal{M}(\theta).
\end{equation}
{ 
	The first equality is the operational definition. The second \revised{states
	that the induced weights solve} the projection onto the curved union of moment slices,
	provided each slice projection exists. Since $\mathcal{M}_\Theta$ is
	generally nonconvex, this is not a convex Bregman projection in the sense
	of Theorem~\ref{thm:main}. Rather, it is a family of exact convex
	slice-projections followed by a nonlinear minimization over $\theta$.}

\paragraph{Pythagoras separates calibration cost from misfit.}
On each fixed slice, Theorem~\ref{thm:main} gives the Pythagorean
decomposition
\[
D_G(\omega\,\|\,\hat P_n)
=
D_G\bigl(\omega\,\|\,\hat\omega(\theta)\bigr)
+
\underbrace{
	D_G\bigl(\hat\omega(\theta)\,\|\,\hat P_n\bigr)
}_{\text{calibration cost at }\theta},
\qquad
\omega\in\mathcal{M}(\theta).
\]
Thus \eqref{eq:overid-est} minimizes the cost of imposing the moment model
at $\theta$. {  The Pythagorean decomposition is exact on each
	affine slice, but it does not extend globally to the curved union
	$\mathcal{M}_\Theta$} \cite{kim2026}.

\paragraph{The conjugate-domain identity.}
Let
\[
\eta^{(0)}=G'(1/n),
\qquad
\hat\eta_i(\theta)
=
G'\bigl(\hat\omega_i(\theta)\bigr)
=
\eta^{(0)}
+
\hat\lambda_\theta^\top U(\theta;Z_i).
\]
Then the calibration cost can be written in the conjugate coordinates as
\begin{equation}
	\label{eq:overid-dual}
	D_G\bigl(\hat\omega(\theta)\,\|\,\hat P_n\bigr)
	=
	\sum_{i=1}^{n}
	D_F\bigl(\eta^{(0)}\,\|\,\hat\eta_i(\theta)\bigr).
\end{equation}
Thus $\hat\theta_G$ selects the parameter whose calibrated natural
coordinates remain closest, in the conjugate divergence $D_F$, to the
constant empirical anchor $\eta^{(0)}$. {  This is the same
	$D_G$/$D_F$ mean--natural pairing that appeared in the GLM section, now
	applied to moment-condition estimation rather than likelihood fitting}
\cite{kim2026}.

\paragraph{One knob over a family of methods.}
The generator selects the member of the method family. For example,
\[
G(\omega)=\frac12(\omega-1/n)^2
\]
gives the quadratic, GMM-type criterion \cite{Hansen1982}\revised{ (with
$I=\R$; as in the calibration case, this member can produce negative
weights)},
\[
G(\omega)=-\log\omega+\omega-1
\]
gives the empirical-likelihood analogue \cite{Owen1988}, and
\[
G(\omega)=\omega\log\omega-\omega+1
\]
gives the exponential-tilting analogue \cite{Kitamura1997}. Because the
simplex constraint $\sum_i\omega_i=1$ has been dropped, these are
unnormalized positive-measure analogues of empirical likelihood and
exponential tilting. Imposing the simplex constraint recovers the standard
generalized empirical likelihood forms \cite{NeweySmith2004}.

\section{Latent-variable inference: EM, variational inference,
	and expectation propagation}
\label{sec:em}

\revised{Many standard latent-variable inference algorithms can be organized in
	terms of the two projections, but they do not all sit at the same distance
	from the exact projection theorem.} The expectation--maximization
algorithm alternates them exactly; when the inference (E-) step is
intractable, the standard remedies --- restricting the variational family,
amortizing it with a network, or trading the reverse-KL projection for its
forward-KL dual --- \revised{can be viewed as variations on} that one step,
organized by the e/m choice. They differ, though, in how literally the
projection theorem applies: EM is an \emph{exact} alternating projection
in the flat exponential-family setting; mean-field and amortized
variational inference are \emph{restricted or curved} approximations, the
family being in general neither e- nor m-flat; \revised{expectation
	propagation runs the dual, m-projection side of the same step;} and
\revised{score matching, flow matching, GANs, and diffusion models belong
	to neighboring training geometries rather than to the exact e/m projection
	theorem itself.} We take them in turn, flagging the status of each.

\subsection{EM as alternating projection}
\label{sec:em-alt}

EM \emph{alternates} two projections in the space of joint distributions
$r(x,z)$, \revised{where $x$ carries the observed data and $z$ the latent
	variable.}
{\revisedblock
	We first record the e-projection that will run the E-step. For a
	prescribed marginal density $\rho(x)$, the set
	\[
	\mathcal{M}(\rho)=\Big\{\,r:\ \textstyle\int r(x,z)\,d\mu(z)=\rho(x)
	\ \ \mu\text{-a.e.}\ x\,\Big\}
	\]
	is m-flat, because the constraints are moments of functions of $x$ alone.
	The e-projection \eqref{eq:primal} of a reference joint $p$, with
	$x$-marginal $p_X$, onto $\mathcal{M}(\rho)$ has the explicit solution
	\[
	q=\arg\min_{r\in\mathcal{M}(\rho)}\KL\big(r\,\|\,p\big)
	=\rho(x)\,p(z\mid x),
	\qquad
	\KL\big(q\,\|\,p\big)=\KL\big(\rho\,\|\,p_X\big).
	\]
	Indeed, it is the Gibbs tilt $q=p\,e^{\lambda(x)}$ of \eqref{eq:soln}
	with $e^{\lambda(x)}=\rho(x)/p_X(x)$\revised{, read with a
	function-valued multiplier, one coordinate per marginal constraint;
	for finite $\Xi$ this is literally \eqref{eq:soln}, and in general the
	displayed solution and value are verified by the chain rule
	$\KL(r\,\|\,p)=\KL(\rho\,\|\,p_X)
	+\E_\rho\big[\KL\big(r(\cdot\mid x)\,\|\,p(\cdot\mid x)\big)\big]$
	for $r\in\mathcal{M}(\rho)$}: the prescribed marginal carries the
	reference conditional, so the conditional factors cancel in the
	divergence and only the marginal discrepancy remains. Since
	$\KL(r\,\|\,p)=-H(r)-\E_r[\log p]$, with
	$H(r)=-\E_r[\log r]$ the Shannon entropy of \S\ref{sec:prelim}, the same
	projection maximizes
	\[
	H(r)+\E_r[\log p]
	\]
	over $\mathcal{M}(\rho)$.
}

{\revisedblock
	Now let the \emph{model manifold} $\mathcal{E}=\{p(x,z;\theta)\}$ be the
	specified working model, e-flat when the complete-data model is an
	exponential family, and let the observed marginal $\tilde p(x)$ supply the
	\emph{data manifold} $\mathcal{M}(\tilde p)$. Estimation seeks a model
	point close to the data manifold, and the objective organizing the search
	is the \emph{free energy}
	\[
	\mathcal{L}(r,\theta)
	=
	\E_r\big[\log p(x,z;\theta)\big]+H(r)
	=
	-\KL\big(r\,\|\,p(\cdot,\cdot\,;\theta)\big).
	\]
	This is the entropy form of the projection above, with the model point as
	reference and $\theta$ now free to vary. At fixed $\theta$, its maximum
	over $r\in\mathcal{M}(\tilde p)$ is
	\[
	\E_{\tilde p}[\log p(x;\theta)]+H(\tilde p),
	\]
	so maximizing $\mathcal{L}$ in both arguments is average maximum
	likelihood. The natural scheme is coordinate ascent --- equivalently,
	alternating minimization of one joint divergence in its two arguments
	\cite{CsiszarTusnady1984} --- passing the \emph{anchor}, the argument held
	fixed, back and forth between the two manifolds.
}

\revised{The EM algorithm is this coordinate ascent written step by step.}
\begin{itemize}
	\item \textbf{E-step $=$ e-projection onto $\revised{\mathcal{M}(\tilde p)}$.}
	Holding $\theta$ fixed, \revised{apply the projection above with the
		current model point as reference and $\rho=\tilde p$:
		\begin{equation}
			\label{eq:estep}
			q^{(t+1)}
			=
			\arg\min_{r\in\mathcal{M}(\tilde p)}
			\KL\big(r\,\|\,p(\cdot,\cdot\,;\theta^{(t)})\big)
			=
			\tilde p(x)\,p(z\mid x;\theta^{(t)}).
		\end{equation}
		This is the posterior attached to the data marginal; the familiar per-$x$
		posterior computation is this projection with the fixed marginal
		suppressed.}
	
	\item \textbf{M-step $=$ m-projection onto $\revised{\mathcal{E}}$.}
	Holding $q=q^{(t+1)}$ fixed, \revised{m-project it onto the model manifold
		by varying the \emph{second} argument:
		\begin{equation}
			\label{eq:mstep}
			p(\cdot,\cdot\,;\theta^{(t+1)})
			=
			\arg\min_{r\in\mathcal{E}}
			\KL\big(q^{(t+1)}\,\|\,r\big).
		\end{equation}
		Since
		\[
		\KL\big(q\,\|\,p(\cdot,\cdot\,;\theta)\big)
		=
		-H(q)-\E_q[\log p(x,z;\theta)]
		\]
		and $H(q)$ is constant in $\theta$, \eqref{eq:mstep} is the familiar
		maximization of the expected complete-data log-likelihood
		$\E_q[\log p(x,z;\theta)]$} --- the m-projection onto the e-flat model
	manifold (Corollary~\ref{cor:mproj}, i.e.\ the e-projection with respect
	to the conjugate $F$).
\end{itemize}

{\revisedblock
	The E-step inherits its exactness from the projection above for any
	specified model; the M-step is the side that needs the exponential-family
	flatness. Namely, \eqref{eq:mstep} is the m-projection of
	Corollary~\ref{cor:mproj} when the model manifold is e-flat and the
	minimizer is attained; otherwise it is a KL minimization over a curved
	family rather than an exact flat projection. Monotonicity is immediate
	from the two anchored steps: the E-step attains the recorded minimum
	$\KL\big(\tilde p\,\|\,p(\cdot\,;\theta^{(t)})\big)$, the M-step can only
	lower the divergence further, and the next E-step minimum is lower still.
	Hence
	\[
	\KL\big(\tilde p\,\|\,p(\cdot\,;\theta^{(t)})\big)
	\]
	is nonincreasing in $t$, equivalently the average log-likelihood
	$\E_{\tilde p}[\log p(x;\theta^{(t)})]$ never decreases. This is the
	free-energy view of \cite{NealHinton1998} and the $e$/$m$ geometry of
	\cite{Amari1995}. For a single observation $x$ the same functional reads
	\[
	\E_q[\log p(x,z;\theta)]+H(q)
	=
	\log p(x;\theta)
	-
	\KL\big(q\,\|\,p(z\mid x;\theta)\big)
	\le
	\log p(x;\theta),
	\]
	the \emph{evidence lower bound}, the ELBO of \S\ref{sec:vi}.
}

Figure~\ref{fig:em} traces the resulting sequence:
\revised{$\mathcal{L}$} increases at each step and, under standard local
regularity, \revised{its limit points are stationary}.

{\revisedblock
	It is worth separating the two assumptions in play. The \emph{algorithm}
	requires only feasibility: the posterior $p(z\mid x;\theta)$ must be
	available to compute the E-step \eqref{eq:estep}, and the M-step
	maximization must be solvable. Granted these, the coordinate ascent and
	likelihood monotonicity above hold for \emph{any} model; no
	exponential-family structure is used for monotonicity itself. Flatness
	enters only on the geometric side, and only for the M-step: it upgrades
	\eqref{eq:mstep} to the exact flat m-projection of
	Corollary~\ref{cor:mproj}, with its Pythagorean relation. The E-step
	\eqref{eq:estep} is an exact e-projection regardless, as shown above.
	When even the posterior is out of reach, the E-step itself must be
	approximated; that failure --- of feasibility, not of flatness --- is what
	the next subsection addresses.
}

\begin{figure}[ht]
	\centering
	\begin{tikzpicture}[scale=0.82, >={Stealth[length=2mm]},
		dot/.style={circle,fill,inner sep=1.1pt}]
		\draw[thick] (2.20,0.45) -- (10.90,2.35);
		\draw[thick] (2.85,4.60) -- (10.90,3.20);
		\node[left] at (2.18,0.45) {$\mathcal{E}$ (model, $e$-flat)};
		\node[left] at (2.83,4.60) {$\mathcal{M}(\tilde p)$ (data, $m$-flat)};
		
		\node[dot] at (3.30,0.690) {};
		\node[dot] at (4.60,4.296) {};
		\node[dot] at (5.80,1.236) {};
		\node[dot] at (6.80,3.913) {};
		\node[dot] at (7.60,1.629) {};
		\node[dot] at (8.25,3.661) {};
		
		\node[below] at (3.30,0.60) {$\theta^{(0)}$};
		\node[above] at (4.60,4.38) {$q^{(1)}$};
		\node[below] at (5.80,1.15) {$\theta^{(1)}$};
		\node[above] at (6.80,4.00) {$q^{(2)}$};
		\node[below] at (7.60,1.54) {$\theta^{(2)}$};
		
		\draw[->] (3.30,0.690) -- (4.60,4.296);
		\draw[->] (4.60,4.296) -- (5.80,1.236);
		\draw[->] (5.80,1.236) -- (6.80,3.913);
		\draw[->] (6.80,3.913) -- (7.60,1.629);
		\draw[->] (7.60,1.629) -- (8.25,3.661);
		\draw[->,dashed] (8.25,3.661) -- (9.34,2.045);
		
		\node[dot] at (9.40,2.022) {};
		\node[below] at (9.40,1.93) {$\theta^{*}$};
		\node[dot] at (9.40,3.461) {};
		\node[above] at (9.40,3.55) {$q^{*}$};
		\draw[dotted,<->,thick] (9.40,2.13) -- (9.40,3.35);
		\node[right, align=left] at (9.55,2.74)
		{\footnotesize residual gap\\[-2pt]\scriptsize $\KL\big(\tilde p\,\|\,p(\cdot\,;\theta^{*})\big)$};
		
		\node[left] at (3.95,2.49) {\footnotesize E-step};
		\node[right] at (5.25,2.77) {\footnotesize M-step};
	\end{tikzpicture}
	\caption{EM as alternating projection. From an initial model point
		$\theta^{(0)}\in\revised{\mathcal{E}}$, each E-step e-projects onto the
		data manifold $\revised{\mathcal{M}(\tilde p)}$ and each M-step m-projects
		onto the model manifold $\revised{\mathcal{E}}$; \revised{in the flat
			exponential-family setting both steps are exact Bregman projections
			satisfying the corresponding Pythagorean relations
			(Theorem~\ref{thm:main} and Corollary~\ref{cor:mproj}); the free
			energy $\mathcal{L}$ increases at every step by the anchored
			coordinate-ascent argument, with or without flatness.}
		\revised{Generically the two manifolds do not intersect: under standard
			regularity the iterates approach a stationary pair ---
			$q^{*}\in\mathcal{M}(\tilde p)$ and
			$p(\cdot;\theta^{*})\in\mathcal{E}$, locally closest points of the two
			manifolds --- separated by the residual gap
			$\KL\big(q^{*}\,\|\,p(\cdot,\cdot\,;\theta^{*})\big)
			=\KL\big(\tilde p\,\|\,p(\cdot\,;\theta^{*})\big)$, the misfit of the
			best-fitting model. The manifolds intersect, and the gap closes, exactly
			when the model is well specified; otherwise the alternating-projection
			picture is a useful local guide.}}
	\label{fig:em}
\end{figure}
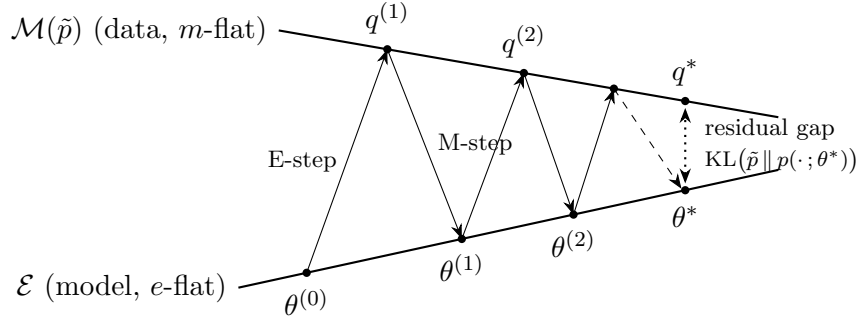

When both the data and model families are flat in the $G/F$ geometry, the
same alternating projections \revised{\eqref{eq:estep}--\eqref{eq:mstep}}
can be run with $D_G$ in place of $\KL$\revised{ --- though the closed
form in \eqref{eq:estep}, with the reference conditional preserved, is
special to Shannon; for a general generator the E-step solution takes
the form of \eqref{eq:soln} and no longer factors through the reference
conditional ---}, giving a Bregman analogue of EM;
its likelihood or evidence interpretation, however, is generator- and
model-dependent\revised{, and for a general generator the monotonicity of a
	Shannon likelihood is no longer automatic}.

\subsection{Variational inference: the inexact E-step}
\label{sec:vi}

Fix the generative model $p(x,z;\theta)=p_0(z)\,p(x\mid z;\theta)$ --- its
functional form specified and $\theta$ held at the current iterate. The
exact E-step computes the posterior
\[
p(z\mid x;\theta)\propto p_0(z)\,p(x\mid z;\theta),
\]
\revised{the conditional part of the e-projection \eqref{eq:estep} onto
	$\mathcal{M}(\tilde p)$}. What is intractable is not the joint, which is
known, but its normalizer,
\[
p(x;\theta)=\int p(x,z;\theta)\,d\mu(z).
\]
One therefore \revised{works per observation --- the projection of
	\S\ref{sec:em-alt} is solved conditional-by-conditional, so nothing is
	lost by fixing one $x$ ---} restricts the candidate $q(z)$ to a tractable
family $\mathcal{Q}$\revised{, such as a parametrized Gaussian family or a
	factorized mean-field family}, and e-projects onto $\mathcal{Q}$ instead:
\[
\min_{q\in\mathcal{Q}}\ \KL\big(q\,\|\,p(z\mid x;\theta)\big)
\quad\Longleftrightarrow\quad
\max_{q\in\mathcal{Q}}\ \big[\E_q\log p(x,z;\theta)+H(q)\big]
=
\max_{q\in\mathcal{Q}}\ELBO(q).
\]
\revised{The objective is the free energy of \S\ref{sec:em-alt} in its
	single-observation form, at fixed $\theta$; \emph{evidence lower bound} is
	the name it carries in this literature, earned by the inequality
	$\ELBO(q)\le\log p(x;\theta)$. The right-hand form involves only the known
	joint and never the intractable marginal, which is what makes $q^{*}$
	computable; the left-hand form shows what the maximization does: it drives
	$q$ toward the posterior.}

With $\mathcal{Q}$ unrestricted the maximizer is the exact posterior, the
Shannon case of the e-projection solution \eqref{eq:soln}, an exponential
tilt of $p_0$\revised{, the prior here playing the role of
	Part~\ref{part:math}'s reference}. It differs from the constrained
e-projections of \S\ref{sec:maxent} and \S\ref{sec:calib} in one respect:
the tilt is the log-likelihood $\log p(x\mid z;\theta)$, given by the
model rather than solved from a moment constraint. Restricting
$\mathcal{Q}$ makes the projection approximate: the residual
\[
\KL\big(q^{*}\,\|\,p(z\mid x;\theta)\big)\ge0,
\]
equivalently the slack in the ELBO, is the cost of the restriction and
vanishes exactly when $\mathcal{Q}$ contains the posterior. A mean-field
$\mathcal{Q}$ makes the projection \revised{decompose coordinatewise
	in $z$}, yielding coordinate-ascent VI updates\revised{ in
	closed form under conditional conjugacy}. \revised{For a parametric family
	$q_\eta$, the maximization proceeds instead by stochastic gradient ascent,
	sampling from $q_\eta$ and differentiating through the samples by
	reparametrization \cite{KingmaWelling2014}; this is the route taken by
	amortized methods.}

\subsection{Amortized inference: autoencoders}
\label{sec:amortized}

{\revisedblock
	The variational step of the preceding subsection has a cost structure
	worth noticing: it solves one optimization per observation, and must solve
	a fresh one for every new $x$. With large data sets, gradient-based inner
	loops, and a model whose $\theta$ is itself being learned, the per-datum
	optimizations dominate. \emph{Amortized} inference removes the inner loop
	by learning the solution map itself: a single network $q_\eta(z\mid x)$,
	trained once, that carries any $x$ to an approximate posterior directly.
	The per-datum inference cost is spread across the data set, and inference
	at a new point becomes a forward pass.
}

\revised{The variational autoencoder realizes this with two networks.} The
\emph{decoder} $p_\theta(x\mid z)$ is the \revised{likelihood factor of
	the} generative model of \S\ref{sec:vi}: it maps a latent $z$ to a
distribution over the data $x$ and plays the M-step role\revised{. Together
	with the prior $p_0(z)$ it defines the joint model, although in a VAE its
	parameters are updated by gradient ascent on the ELBO, an inexact
	coordinate update on the same functional rather than a closed-form
	M-step}. The \emph{encoder} $q_\eta(z\mid x)$ \revised{is the learned
	inference map}, standing in for the E-step. The two networks are trained
jointly by maximizing the ELBO\revised{, summed over the data}.

This \revised{has} the same impute--update skeleton as the data
augmentation algorithm \cite{Tanner1987}, the stochastic counterpart of
EM. There one alternates between drawing
$z\sim p(z\mid x;\theta)$ --- imputing the missing data --- and updating
$\theta$ given the completed data $(x,z)$. The encoder plays the
imputation step and the decoder the complete-data model; the difference is
that data augmentation \emph{samples} the exact posterior, asymptotically
exact but per-datum and iterative, while the encoder \emph{learns} one
parametric map that approximates it, amortized over the data set but only
as accurate as the encoder family allows.

How the encoder is fit is itself the e/m choice: the standard VAE
\cite{KingmaWelling2014,Rezende2014} minimizes the reverse divergence
\[
\KL\big(q_\eta(z\mid x)\,\|\,p_\theta(z\mid x)\big),
\]
an e-projection\revised{, with the inference gap optimized through the
	ELBO even though the posterior itself is unavailable}. By contrast, the
sleep phase of wake--sleep \cite{HintonWakeSleep1995} minimizes the
forward divergence
\[
\KL\big(p_\theta(z\mid x)\,\|\,q_\eta(z\mid x)\big),
\]
an m-projection\revised{, estimated under samples drawn from the model}.
This is the amortized, parametric corner: the encoder family is generally
neither e- nor m-flat, so the exact uniqueness and Pythagorean identity of
\S\ref{sec:projections} need not hold. The autoencoder situates practice
against the exact projections rather than being governed by the theorems.

Generalizing the encoder's divergence beyond KL has so far been pursued
mainly along the $f$-divergence family --- R\'enyi \cite{LiTurner2016},
$\chi^{2}$ \cite{Dieng2017}, and general $f$-divergences \cite{Wan2020}
--- which meets the Bregman family only at KL. Fitting the encoder by a
general $D_G$ would be the amortized Bregman EM \revised{of
	\S\ref{sec:em-alt}}, a direction we return to in \S\ref{sec:future}.

\subsection{Expectation propagation: the m-projection dual}
\label{sec:ep}

Expectation propagation runs the inference step the other way. It
approximates the posterior by an exponential family through repeated local
m-projections that match the moments of the tilted distribution
(Corollary~\ref{cor:mproj}); a single such moment match, without
iteration, is formally analogous to method-of-moments estimation. Where
variational inference e-projects, using reverse KL and therefore
\revised{roughly} mode-seeking behavior, EP \revised{locally} m-projects,
using forward KL and therefore \revised{roughly} mean-seeking behavior ---
the same e/m choice, applied to the inference step.

{\revisedblock
	The mechanics make the projection content exact \cite{Minka2001}. Suppose
	the posterior factors into sites
	\[
	p(z\mid x)\propto p_0(z)\prod_{j=1}^{J}\ell_j(z),
	\]
	the prior times one likelihood factor per observation or block, with the
	latent variable $z$ now shared across them as for a global parameter. EP
	maintains an exponential-family approximation
	\[
	q(z)\propto p_0(z)\prod_j\tilde\ell_j(z),
	\]
	in which each true site $\ell_j$ is replaced by an exponential-family
	factor $\tilde\ell_j$, and cycles through the sites: delete one
	approximate site to form the \emph{cavity}, tilt the cavity by the true
	site, and m-project the tilted distribution back onto the family,
	\[
	q_{-j}\propto\frac{q}{\tilde\ell_j},
	\qquad
	r_j\propto q_{-j}\,\ell_j,
	\qquad
	q^{\mathrm{new}}
	=
	\arg\min_{q'\in\mathcal{E}}\KL\big(r_j\,\|\,q'\big)
	\quad\Longleftrightarrow\quad
	\E_{q^{\mathrm{new}}}[T]=\E_{r_j}[T].
	\]
	The new site is recovered as
	\[
	\tilde\ell_j\propto q^{\mathrm{new}}/q_{-j},
	\]
	and here $\mathcal{E}$ is an exponential family over the latent $z$, not
	the joint model manifold of \S\ref{sec:em-alt}. The tilt has the same
	structure as the tilt-of-the-reference reading of the posterior in
	\S\ref{sec:vi}: the cavity plays the reference $p_0$, and one likelihood
	factor plays the tilt. The projection step is the moment match of
	Corollary~\ref{cor:mproj}; onto the e-flat family $\mathcal{E}$, the
	forward-KL minimization is solved exactly by matching sufficient
	statistics.
}

Each site update is therefore an \emph{exact} flat projection; what is
heuristic is the loop around it. The iteration minimizes no single
divergence --- its fixed points are stationary points of an EP energy
function rather than minimizers of
$\KL\big(p(\cdot\mid x)\,\|\,q\big)$ \cite{Minka2005}. Thus EP's status in
the taxonomy of this section is the mirror image of variational
inference's: VI optimizes a global bound over a restricted family on the
e-side; EP performs locally exact projections without a global objective
on the m-side. The two directions are bridged by \emph{power EP}, which
replaces the forward KL in the projection step by an
$\alpha$-divergence: $\alpha=1$ recovers the EP moment match,
$\alpha\to0$ the reverse-KL update of variational inference, and
intermediate $\alpha$ interpolates \cite{Minka2005}. \revised{This is the
	same $\alpha$/$f$-divergence dial met in the encoder generalizations above,
	now turned from the message-passing side.}

\subsection{Deep generative models and neighboring Bregman geometries}
\label{sec:deep-generative-bregman}

{\revisedblock
	The preceding examples locate several deep generative methods relative to
	the Bregman projection framework. The guiding distinction is the object on
	which the training discrepancy is imposed. In the exact projection theory
	of Part~\ref{part:math}, the objects are densities or measures, and the
	geometry is governed by a separable Bregman divergence together with its
	e/m Pythagorean theorem. In modern deep generative modeling, the same
	theme reappears in several neighboring forms: KL projections for
	variational autoencoders, quadratic losses on score fields for
	score-based diffusion models, quadratic losses on velocity fields for flow
	matching, and $f$-divergence or optimal-transport discrepancies for
	adversarial models.
}

The variational autoencoder lies closest to the projection framework.
\revised{As explained in \S\ref{sec:amortized},} its encoder
$q_\eta(z\mid x)$ is trained through the reverse-KL term in the evidence
lower bound:
\[
\log p_\theta(x)-\ELBO(\eta,\theta;x)
=
\KL\big(q_\eta(z\mid x)\,\|\,p_\theta(z\mid x)\big).
\]
Thus the encoder performs an amortized, restricted e-projection toward the
posterior, in the Shannon Bregman divergence. The decoder update plays the
role of an inexact M-step. The exact Pythagorean theorem need not apply,
because the encoder and decoder are usually curved neural-network
families rather than flat statistical submanifolds, but the KL/Bregman
interpretation is direct \cite{KingmaWelling2014,Rezende2014}.

Score matching sits one step farther away. It does not project densities
onto flat families. Instead, it compares the \emph{score fields}
\[
s_\theta(x)=\nabla_x\log p_\theta(x),
\qquad
s_{\rm data}(x)=\nabla_x\log p_{\rm data}(x),
\]
through the Fisher-divergence-type quadratic criterion
\[
\frac12
\E_{p_{\rm data}}
\big[
\|s_\theta(X)-s_{\rm data}(X)\|^2
\big].
\]
This is a Bregman divergence for the quadratic functional
\[
\mathcal{G}(s)
=
\frac12
\E_{p_{\rm data}}
\big[
\|s(X)\|^2
\big]
\]
on a Hilbert space of square-integrable vector fields:
\[
D_{\mathcal{G}}(s\|t)
=
\frac12
\E_{p_{\rm data}}
\big[
\|s(X)-t(X)\|^2
\big].
\]
Thus score matching is a quadratic Bregman fitting problem, but the
Bregman geometry is placed on the space of scores rather than directly on
the space of densities.

The practical advantage is that the unknown normalizing constant of
$p_\theta$ disappears after taking $\nabla\log p_\theta$. Hyv\"arinen's
integration-by-parts identity rewrites the ideal loss, up to an additive
constant independent of $\theta$, as
\[
\E_{p_{\rm data}}
\left[
\operatorname{div}s_\theta(X)
+
\frac12\|s_\theta(X)\|^2
\right],
\]
under suitable boundary conditions \cite{Hyvarinen2005}. Denoising score
matching and score-based diffusion models use the same quadratic geometry
after perturbing the data by noise, fitting time-indexed scores
$s_\theta(t,x)$ for the marginals of a forward noising process
\cite{Vincent2011,SongErmon2019,Ho2020,Song2021}. Hence score-based
diffusion is naturally adjacent to the Bregman framework, but its
Bregman divergence lives on score fields, not directly on densities.

Flow matching is analogous, with velocity fields replacing score fields.
One specifies or samples a probability path $(p_t)_{0\le t\le1}$ from a
simple base distribution to the data distribution and fits a vector field
$v_\theta(t,x)$ to a target field $u_t(x)$ satisfying the continuity
equation
\[
\partial_t p_t+\nabla\cdot(p_tu_t)=0.
\]
The training loss is quadratic:
\[
\frac12
\E_{t,X_t}
\big[
\|v_\theta(t,X_t)-u_t(X_t)\|^2
\big].
\]
This is exactly the Bregman divergence generated by the quadratic energy
on a function space. Indeed, define
\[
\mathcal{G}(v)
=
\frac12
\E_{t,X_t}
\big[
\|v(t,X_t)\|^2
\big].
\]
Then $\nabla\mathcal{G}(u)=u$ with respect to the
$L^2(dt\,p_t(dx))$ pairing, and therefore
\[
\begin{aligned}
	D_{\mathcal{G}}(v\|u)
	&=
	\mathcal{G}(v)-\mathcal{G}(u)
	-
	\langle \nabla\mathcal{G}(u),v-u\rangle  \\
	&=
	\frac12\E_{t,X_t}\|v(t,X_t)\|^2
	-
	\frac12\E_{t,X_t}\|u(t,X_t)\|^2
	-
	\E_{t,X_t}\langle u(t,X_t),v(t,X_t)-u(t,X_t)\rangle \\
	&=
	\frac12
	\E_{t,X_t}
	\big[
	\|v(t,X_t)-u(t,X_t)\|^2
	\big].
\end{aligned}
\]
Consequently the flow-matching objective is
$D_{\mathcal{G}}(v_\theta\|u)$ for the quadratic generator
$\mathcal{G}$ on velocity fields.

{\revisedblock
	This clarification is important. Flow matching is not a Bregman
	projection of densities. The model is not obtained by minimizing
	$D_G(p\|p_0)$ over an $m$-flat constraint set, nor by minimizing
	$D_G(p_0\|p)$ over an $e$-flat family. Rather, a dynamical transport
	problem is converted into regression of vector fields; the induced
	evolution of densities is governed by the continuity equation, while the
	Bregman object is the squared-error divergence on velocity fields
	\cite{Lipman2023,Albergo2023,Albergo2023StochasticInterpolants,Liu2023FlowStraight}.
	Thus flow matching is close in spirit to the projection framework, but it
	is best classified as a neighboring quadratic-Bregman method rather than
	an exact e/m Bregman projection.
}

Adversarial models occupy a different neighboring region. The original
GAN minimizes, at the optimal discriminator, a Jensen--Shannon divergence
between the data distribution and the model distribution
\cite{Goodfellow2014}. The $f$-GAN formulation replaces this by a general
$f$-divergence \cite{Nowozin2016}, while the Wasserstein GAN uses the
Wasserstein distance, an optimal-transport metric rather than a Bregman
or $f$-divergence \cite{Arjovsky2017}. These methods are therefore best
placed on the $f$-divergence and optimal-transport sides of the landscape,
not inside the Bregman projection theorem.

\paragraph{Kernel discrepancies and MMD estimation.}

Another neighboring region is formed by kernel discrepancy methods for
implicit generative models.  Suppose that \(P_\theta\) is a generative
model from which simulation is easy but whose likelihood is unavailable
or intractable.  \cite{briol2019} propose to estimate
\(\theta\) by minimizing the maximum mean discrepancy
\[
        \operatorname{MMD}_k^2(P_\theta,\widehat P_n)
        =
        \|\mu_{P_\theta}-\mu_{\widehat P_n}\|_{\mathcal H_k}^2
\]
between the model distribution and the empirical distribution, where
\(\mu_P\) denotes the kernel mean embedding of \(P\) in the reproducing
kernel Hilbert space \(\mathcal H_k\).  This is not a Bregman projection
of densities in the sense of Part I: the discrepancy is imposed after
embedding probability measures into an RKHS, rather than through a
separable Bregman divergence on the densities themselves.  Nevertheless,
it has the same statistical role as the neighboring generative-model
objectives discussed above: likelihood is replaced by a geometrically
defined discrepancy, and estimation becomes a minimum-distance problem
over a model family.  In this sense, MMD estimation sits alongside
score matching and flow matching as a neighboring geometry for
generative-model estimation, with the discrepancy living on kernel mean
embeddings rather than on densities, scores, or velocity fields.

\section{Directions for further investigation}
\label{sec:future}

\revised{We close Part~\ref{part:stat} with directions the geometry leaves
open.}

\paragraph{Robustness through the generator.}
Power and Tsallis generators give range-restricted projections
(Remark~\ref{rem:a4}); the generator $G$ is a tuning knob trading
efficiency for \revised{outlier- and heavy-tail-robustness} whose statistical
consequences merit systematic study. \revised{For the exponential-tilting
generator, for instance, the resulting estimator retains $\sqrt{n}$-consistency
under model misspecification --- a global robustness the empirical-likelihood
generator can lose when the moment functions are unbounded
\cite{Schennach2007} --- one concrete entry in this trade-off.} \revised{Other
members occupy other corners: the local-robustness optimum is the Hellinger
member --- minimum-Hellinger-distance estimation is minimax-robust over
shrinking (infinitesimal) contamination neighborhoods
\cite{KitamuraOtsuEvdokimov2013}, sitting in the $f$-divergence direction
rather than the Bregman family, the two meeting only at the KL point occupied
by exponential tilting and empirical likelihood (\S\ref{sec:calib}) --- while
among higher-order refinements the empirical-likelihood member is
Bartlett-correctable \cite{DiCiccioHallRomano1991} and exponential tilting is
not \cite{JingWood1996}. No single generator is best across efficiency, global
and local robustness, and higher-order accuracy at once.} This raises the question of how to \emph{choose} $G$
from data. The calibration cost
$D_G(\hat\omega(\hat\theta_G)\,\|\,\hat P_n)$ of \S\ref{sec:calib} is a
natural goodness-of-fit criterion, but since $\hat\theta_G$ already
minimizes it, comparing generators by its in-sample value double-dips;
evaluating it by sample-splitting --- fitting the weights and
$\hat\theta_G$ on one fold and scoring the calibration cost on a held-out
fold, with cross-fitting \cite{Chernozhukov2018} to recover efficiency ---
gives an honest criterion, informative mainly under misspecification or
contamination (under correct specification the rescaled cost is
first-order generator-invariant). This parallels the data-driven selection
of the tuning parameter of the density-power divergence \cite{Basu1998}
--- itself a separable Bregman divergence --- by an estimated
mean-squared-error criterion \cite{WarwickJones2005,BasakBasuJones2021},
with the held-out calibration cost as a fit-based alternative.

\paragraph{Identifiability under under-determined constraints.}
When the moments do not pin down the target, \revised{the estimand is not
identified --- distinct feasible distributions carry it to different
values ---} \revised{as is} the rule in missing-not-at-random and data-integration
problems. \revised{The projection still returns a single point (Theorem~\ref{thm:main}), so
the choice of $G$ selects one value within the identified set;
characterizing that set, and how $G$ selects within it, is largely open.}

\paragraph{Efficiency theory.}
The dual stationarity $\E_{p^{*}_\lambda}[T]=c$ is a $Z$-estimating
equation \revised{whose estimating function has the generalized information matrix
of Proposition~\ref{prop:est} as its Jacobian}; a full asymptotic and
semiparametric-efficiency theory in this language would unify scattered
results.

\paragraph{A general-Bregman encoder.}
The variational autoencoder of \revised{\S\ref{sec:amortized}} fits its encoder by a
reverse-KL projection. Replacing that by a general separable Bregman
divergence $D_G(q\,\|\,p)$ --- the amortized form of the Bregman EM --- is
largely unexplored: the popular non-KL encoders (R\'enyi, $\chi^{2}$, and
$f$-divergence variational inference) generalize along the $f$-divergence
family, which intersects the Bregman family only at KL \cite{Wan2020},
while the Bregman direction has so far appeared mainly in robust
$\beta$-divergence inference \cite{Futami2018} and the loss-plus-divergence
framework of generalized variational inference \cite{Knoblauch2019}.
Characterizing which generators $G$ yield well-posed, identifiable, or
outlier-robust amortized inference would carry the geometry of
Part~\ref{part:math} into deep generative models.

\section{Summary}
\label{sec:summary}

\begin{table}[h]
	\centering
	\small
	\resizebox{1.1\textwidth}{!}{%
		\begin{tabular}{@{}llll@{}}
			\toprule
			Method & Projection / discrepancy & Generator / geometry & Status \\
			\midrule
			GLM / maximum likelihood
			& e- (natural) $=$ m- (mean)
			& cumulant $F$
			& exact (canonical) \\
			
			Maximum entropy / MLE
			& e- \revised{and} m- (intersection)
			& Shannon
			& exact \\
			
			Survey calibration
			& e-projection (one slice)
			& Shannon, quadratic, \revised{or general $G$}
			& exact \\
			
			Over-identified model
			& e- per slice, \revised{projection over curved union}
			& any \revised{admissible $G$} (GMM, EL, ET)
			& exact per slice / asymptotic \\
			
			EM
			& alternating e/m
			& Shannon
			& exact if flat \\
			
			Variational inference
			& restricted e-projection
			& Shannon
			& curved approximation \\
			
			\revised{Variational autoencoder}
			& \revised{amortized restricted e-projection}
			& \revised{Shannon / KL}
			& \revised{amortized approximation} \\
			
			Expectation propagation
			& m-projection
			& Shannon
			& \revised{locally exact; no global objective} \\
			
			\revised{Score matching / diffusion}
			& \revised{quadratic loss on score fields}
			& \revised{quadratic Bregman geometry}
			& \revised{neighboring field-level analogue} \\
			
			\revised{Flow matching}
			& \revised{quadratic loss on velocity fields}
			& \revised{quadratic Bregman geometry}
			& \revised{neighboring field-level analogue} \\
			
			\revised{GAN / $f$-GAN / WGAN}
			& \revised{$f$-divergence or transport discrepancy}
			& \revised{$f$-geometry / optimal transport}
			& \revised{outside Bregman projection theorem} \\
			\bottomrule
		\end{tabular}%
	}
	\caption{The statistical methods of Part~\ref{part:stat} organized by the
		projection or discrepancy they use. The final column flags whether the
		identification is an exact Bregman projection, exact only under
		flatness/regularity, a restricted or amortized approximation, or a
		neighboring divergence analogy.}
\end{table}

The whole picture is one construction. A generator $G$ and its conjugate
$F$ give two coordinate systems; a reference is projected onto a flat set;
the Pythagorean theorem \eqref{eq:pyth} holds and, applied through $G$ or
through $F$, produces the e- and m-projections (Theorem~\ref{thm:main},
Corollary~\ref{cor:mproj}, Proposition~\ref{prop:swap}). The
\revised{canonical-link} generalized linear model shows the two
projections concretely --- the score equation is the Pythagorean
orthogonality, the fit is an e-projection in the natural coordinate and an
m-projection in the mean coordinate.

The remaining methods are organized by the same construction, but not all
literally. In the flat cases --- maximum entropy, calibration, and EM on
an exponential family --- they are exact Bregman projections obeying the
Pythagorean theorem. \revised{Expectation propagation sits between the exact
	and approximate categories: each site update is an exact flat
	m-projection, but the loop around those local projections optimizes no
	single global divergence.} In the restricted or amortized cases ---
mean-field variational inference and variational autoencoders --- the
same KL/Bregman projection is performed only within a curved or
neural-network family, so the exact Pythagorean theorem becomes a guiding
geometry rather than a theorem controlling the fitted point.

\revised{The modern generative-model examples lie one step farther away.
	Score matching, denoising diffusion, and flow matching do not project
	densities onto flat Bregman families. Instead, they move the discrepancy
	to an auxiliary function space: score fields for score matching and
	velocity fields for flow matching. Their squared-error objectives are
	Bregman divergences for quadratic generators on those function spaces, so
	they are neighboring quadratic-Bregman methods rather than exact e/m
	density projections. GANs, $f$-GANs, and Wasserstein GANs move farther
	still, replacing the Bregman projection geometry by $f$-divergence or
	optimal-transport discrepancies.}

\revised{Thus the point of the paper is not that every method is literally
	the same projection. Rather, the Bregman projection theorem supplies the
	organizing center: exact flat projections occupy the core; restricted,
	amortized, and local algorithms are controlled approximations around it;
	and score-, flow-, adversarial-, and transport-based generative models sit
	in neighboring geometries determined by the object on which their training
	discrepancy is imposed.}

\appendix
\section{Notation}

{ 
	Throughout the paper, $G$ denotes the primal Bregman generator and
	$F=G^{*}$ its convex conjugate. The pointwise and separable Bregman
	divergences generated by $G$ are denoted by $d_G$ and $D_G$, respectively,
	and $H_G$ denotes the associated generalized entropy.
}

{ 
	The dual coordinate maps are
}
\[
\theta=\nabla G(p),
\qquad
p=\nabla F(\theta).
\]
{ 
	Thus $p$ is the m-coordinate and $\theta$ is the e-coordinate. The maps
	$\nabla G$ and $\nabla F$ are called the link and inverse link,
	respectively.}

{ 
	The symbol $\mathcal{M}$ denotes a moment, or m-flat, set; when the target
	moment needs to be displayed, we write $\mathcal{M}(c)$. The symbol
	$\mathcal{E}$ denotes an exponential, or e-flat, family. The statistic is
	$T$, the target moment is $c$, the Lagrange multiplier is $\lambda$, and
	$\mathcal{N}$ denotes the constraint null space. The dual objective is
	denoted by \revised{$\Phi$}, and the optimal dual value is denoted by
	$\mathcal{Z}_G$.}

{ 
	For generalized linear models, $F$ denotes the cumulant, viewed as the
	natural-coordinate generator, and $G=F^{*}$ denotes the corresponding
	mean-domain generator. The natural parameter is $\theta=X\beta$, the model
	family in natural coordinates is $\mathcal{E}_X$, the mean is
	$\mu=\nabla F(\theta)$, the saturated natural parameter is $\theta_y$, and
	the fitted natural and mean parameters are $\hat\theta$ and $\hat\mu$.}

{ 
	For survey calibration, $\mathcal{U}$ denotes the finite population,
	$N=|\mathcal{U}|$ its size, $A\subset\mathcal{U}$ the sample, and $n=|A|$
	the sample size. The first-order inclusion probability is $\pi_i$, the
	design weight is $d_i=\pi_i^{-1}$, the calibrated weight is $w_i$, the
	known auxiliary total is $X=\sum_{i\in\mathcal{U}}x_i$, and the study total
	is $Y=\sum_{i\in\mathcal{U}}y_i$.}

{ 
	For over-identified moment models, $Z_i$ denotes the observed data,
	$U(\theta;Z_i)$ the estimating function, $r$ the number of moment
	conditions, $p$ the dimension of the parameter $\theta$, $\hat P_n$ the
	empirical anchor, $\omega$ the positive sample-weight vector,
	$\mathcal{M}(\theta)$ the fixed-$\theta$ moment slice, and
	$\mathcal{M}_\Theta=\bigcup_{\theta\in\Theta}\mathcal{M}(\theta)$ the
	curved union of moment slices.}

\bibliographystyle{plain}
\bibliography{references_2}

@inproceedings{Rezende2014,
	author    = {Danilo Jimenez Rezende and Shakir Mohamed and Daan Wierstra},
	title     = {Stochastic Backpropagation and Approximate Inference in Deep Generative Models},
	booktitle = {Proceedings of the 31st International Conference on Machine Learning},
	series    = {Proceedings of Machine Learning Research},
	volume    = {32},
	number    = {2},
	pages     = {1278--1286},
	year      = {2014},
	publisher = {PMLR}
}

@article{Vincent2011,
	author  = {Pascal Vincent},
	title   = {A Connection Between Score Matching and Denoising Autoencoders},
	journal = {Neural Computation},
	volume  = {23},
	number  = {7},
	pages   = {1661--1674},
	year    = {2011}
}

@inproceedings{SongErmon2019,
	author    = {Yang Song and Stefano Ermon},
	title     = {Generative Modeling by Estimating Gradients of the Data Distribution},
	booktitle = {Advances in Neural Information Processing Systems (NeurIPS)},
    pages   = {11918--11930},
	year      = {2019}
}

@inproceedings{Lipman2023,
	author    = {Yaron Lipman and Ricky T. Q. Chen and Heli Ben-Hamu and Maximilian Nickel and Matt Le},
	title     = {Flow Matching for Generative Modeling},
	booktitle = {International Conference on Learning Representations (ICLR)},
	year      = {2023}
}

@inproceedings{Albergo2023,
	author    = {Michael S. Albergo and Eric Vanden-Eijnden},
	title     = {Building Normalizing Flows with Stochastic Interpolants},
	booktitle = {International Conference on Learning Representations (ICLR)},
	year      = {2023}
}

@article{Albergo2023StochasticInterpolants,
	author  = {Michael S. Albergo and Nicholas M. Boffi and Eric Vanden-Eijnden},
	title   = {Stochastic Interpolants: A Unifying Framework for Flows and Diffusions},
	journal = {Journal of Machine Learning Research},
  year    = {2025},
  volume  = {26},
  number  = {209},
  pages   = {1--80}
}

@inproceedings{Liu2023FlowStraight,
	author    = {Xingchao Liu and Chengyue Gong and Qiang Liu},
	title     = {Flow Straight and Fast: Learning to Generate and Transfer Data with Rectified Flow},
	booktitle = {International Conference on Learning Representations (ICLR)},
	year      = {2023}
}

@article{Csiszar1975,
  author  = {Imre Csisz\'ar},
  title   = {{$I$}-divergence geometry of probability distributions and minimization problems},
  journal = {The Annals of Probability},
  volume  = {3},
  number  = {1},
  pages   = {146--158},
  year    = {1975}
}

@article{CsiszarTusnady1984,
  author  = {Imre Csisz\'ar and G\'abor Tusn\'ady},
  title   = {Information geometry and alternating minimization procedures},
  journal = {Statistics and Decisions},
  note    = {Supplement Issue 1, 205--237},
  year    = {1984}
}

@book{AmariNagaoka,
  author    = {Shun-ichi Amari and Hiroshi Nagaoka},
  title     = {Methods of Information Geometry},
  publisher = {American Mathematical Society and Oxford University Press},
  year      = {2000}
}

@article{Amari1995,
  author  = {Shun-ichi Amari},
  title   = {Information geometry of the {EM} and em algorithms for neural networks},
  journal = {Neural Networks},
  volume  = {8},
  number  = {9},
  pages   = {1379--1408},
  year    = {1995}
}

@article{Zhang2004,
  author  = {Jun Zhang},
  title   = {Divergence function, duality, and convex analysis},
  journal = {Neural Computation},
  volume  = {16},
  number  = {1},
  pages   = {159--195},
  year    = {2004}
}

@article{Eguchi,
  author  = {Noboru Murata and Takashi Takenouchi and Takafumi Kanamori and Shinto Eguchi},
  title   = {Information geometry of {U}-boost and {B}regman divergence},
  journal = {Neural Computation},
  volume  = {16},
  number  = {7},
  pages   = {1437--1481},
  year    = {2004}
}

@article{Banerjee2005,
  author  = {Arindam Banerjee and Srujana Merugu and Inderjit S. Dhillon and Joydeep Ghosh},
  title   = {Clustering with {B}regman divergences},
  journal = {Journal of Machine Learning Research},
  volume  = {6},
  number  = {58},
  pages   = {1705--1749},
  year    = {2005}
}

@article{NelderWedderburn1972,
  author  = {John A. Nelder and Robert W. M. Wedderburn},
  title   = {Generalized linear models},
  journal = {Journal of the Royal Statistical Society, Series A},
  volume  = {135},
  number  = {3},
  pages   = {370--384},
  year    = {1972}
}

@book{McCullaghNelder1989,
  author    = {Peter McCullagh and John A. Nelder},
  title     = {Generalized Linear Models},
  edition   = {2nd},
  publisher = {Chapman \& Hall/CRC},
  year      = {1989}
}

@incollection{NealHinton1998,
  author    = {Radford M. Neal and Geoffrey E. Hinton},
  title     = {A view of the {EM} algorithm that justifies incremental, sparse, and other variants},
  booktitle = {Learning in Graphical Models},
  editor    = {Michael I. Jordan},
  pages     = {355--368},
  publisher = {Kluwer Academic Publishers},
  year      = {1998}
}

@inproceedings{KingmaWelling2014,
  author    = {Diederik P. Kingma and Max Welling},
  title     = {Auto-encoding variational {B}ayes},
  booktitle = {International Conference on Learning Representations (ICLR)},
  year      = {2014}
}

@article{Tanner1987,
  author  = {Martin A. Tanner and Wing Hung Wong},
  title   = {The calculation of posterior distributions by data augmentation},
  journal = {Journal of the American Statistical Association},
  volume  = {82},
  number  = {398},
  pages   = {528--540},
  year    = {1987}
}

@article{HintonWakeSleep1995,
  author  = {Geoffrey E. Hinton and Peter Dayan and Brendan J. Frey and Radford M. Neal},
  title   = {The wake-sleep algorithm for unsupervised neural networks},
  journal = {Science},
  volume  = {268},
  number  = {5214},
  pages   = {1158--1161},
  year    = {1995}
}

@inproceedings{LiTurner2016,
  author    = {Yingzhen Li and Richard E. Turner},
  title     = {R\'enyi divergence variational inference},
  booktitle = {Advances in Neural Information Processing Systems (NeurIPS)},
  pages   = {1081--1089},
  year      = {2016}
}

@inproceedings{Dieng2017,
  author    = {Adji B. Dieng and Dustin Tran and Rajesh Ranganath and John Paisley and David M. Blei},
  title     = {Variational inference via {$\chi$} upper bound minimization},
  booktitle = {Advances in Neural Information Processing Systems (NeurIPS)},
  pages     = {2729--2738},
  year      = {2017}
}

@inproceedings{Wan2020,
  author    = {Neng Wan and Dapeng Li and Naira Hovakimyan},
  title     = {{$f$}-divergence variational inference},
  booktitle = {Advances in Neural Information Processing Systems (NeurIPS)},
  pages     = {17370--17379},
  year      = {2020}
}

@inproceedings{Futami2018,
  author    = {Futoshi Futami and Issei Sato and Masashi Sugiyama},
  title     = {Variational inference based on robust divergences},
  booktitle = {International Conference on Artificial Intelligence and Statistics (AISTATS)},
  series    = {Proceedings of Machine Learning Research},
  volume    = {84},
  pages     = {813--822},
  year      = {2018}
}

@article{Knoblauch2019,
  author  = {Jeremias Knoblauch and Jack Jewson and Theodoros Damoulas},
  title   = {Generalized variational inference: Three arguments for deriving new posteriors},
  journal = {arXiv preprint arXiv:1904.02063},
  year    = {2019}
}

@inproceedings{Goodfellow2014,
  author    = {Ian J. Goodfellow and Jean Pouget-Abadie and Mehdi Mirza and Bing Xu and David Warde-Farley and Sherjil Ozair and Aaron Courville and Yoshua Bengio},
  title     = {Generative adversarial nets},
  booktitle = {Advances in Neural Information Processing Systems (NeurIPS)},
  pages     = {2672--2680},
  year      = {2014}
}

@inproceedings{Nowozin2016,
  author    = {Sebastian Nowozin and Botond Cseke and Ryota Tomioka},
  title     = {{$f$}-{GAN}: Training generative neural samplers using variational divergence minimization},
  booktitle = {Advances in Neural Information Processing Systems (NeurIPS)},
  pages     = {271--279},
  year      = {2016}
}

@inproceedings{Arjovsky2017,
  author    = {Martin Arjovsky and Soumith Chintala and L\'eon Bottou},
  title     = {{W}asserstein generative adversarial networks},
  booktitle = {International Conference on Machine Learning (ICML)},
  series    = {Proceedings of Machine Learning Research},
  volume    = {70},
  pages     = {214--223},
  year      = {2017}
}

@inproceedings{Ho2020,
  author    = {Jonathan Ho and Ajay Jain and Pieter Abbeel},
  title     = {Denoising diffusion probabilistic models},
  booktitle = {Advances in Neural Information Processing Systems (NeurIPS)},
  pages     = {6840--6851},
  year      = {2020}
}

@inproceedings{Song2021,
  author    = {Yang Song and Jascha Sohl-Dickstein and Diederik P. Kingma and Abhishek Kumar and Stefano Ermon and Ben Poole},
  title     = {Score-based generative modeling through stochastic differential equations},
  booktitle = {International Conference on Learning Representations (ICLR)},
  year      = {2021}
}

@article{Hyvarinen2005,
  author  = {Aapo Hyv\"arinen},
  title   = {Estimation of non-normalized statistical models by score matching},
  journal = {Journal of Machine Learning Research},
  volume  = {6},
  number = {24},
  pages   = {695--709},
  year    = {2005}
}

@article{DevilleSarndal1992,
  author  = {Jean-Claude Deville and Carl-Erik S\"arndal},
  title   = {Calibration estimators in survey sampling},
  journal = {Journal of the American Statistical Association},
  volume  = {87},
  number  = {418},
  pages   = {376--382},
  year    = {1992}
}

@article{Hansen1982,
  author  = {Lars Peter Hansen},
  title   = {Large sample properties of generalized method of moments estimators},
  journal = {Econometrica},
  volume  = {50},
  number  = {4},
  pages   = {1029--1054},
  year    = {1982}
}

@article{Owen1988,
  author  = {Art B. Owen},
  title   = {Empirical likelihood ratio confidence intervals for a single functional},
  journal = {Biometrika},
  volume  = {75},
  number  = {2},
  pages   = {237--249},
  year    = {1988}
}

@article{Schennach2007,
  author  = {Susanne M. Schennach},
  title   = {Point estimation with exponentially tilted empirical likelihood},
  journal = {The Annals of Statistics},
  volume  = {35},
  number  = {2},
  pages   = {634--672},
  year    = {2007}
}

@article{KitamuraOtsuEvdokimov2013,
  author  = {Yuichi Kitamura and Taisuke Otsu and Kirill Evdokimov},
  title   = {Robustness, infinitesimal neighborhoods, and moment restrictions},
  journal = {Econometrica},
  volume  = {81},
  number  = {3},
  pages   = {1185--1201},
  year    = {2013}
}

@article{DiCiccioHallRomano1991,
  author  = {Thomas J. DiCiccio and Peter Hall and Joseph P. Romano},
  title   = {Empirical likelihood is {Bartlett}-correctable},
  journal = {The Annals of Statistics},
  volume  = {19},
  number  = {2},
  pages   = {1053--1061},
  year    = {1991}
}

@article{JingWood1996,
  author  = {Bing-Yi Jing and Andrew T. A. Wood},
  title   = {Exponential empirical likelihood is not {Bartlett} correctable},
  journal = {The Annals of Statistics},
  volume  = {24},
  number  = {1},
  pages   = {365--369},
  year    = {1996}
}

@article{Kitamura1997,
  author  = {Yuichi Kitamura and Michael Stutzer},
  title   = {An information-theoretic alternative to generalized method of moments estimation},
  journal = {Econometrica},
  volume  = {65},
  number  = {4},
  pages   = {861--874},
  year    = {1997}
}

@article{NeweySmith2004,
  author  = {Whitney K. Newey and Richard J. Smith},
  title   = {Higher order properties of {GMM} and generalized empirical likelihood estimators},
  journal = {Econometrica},
  volume  = {72},
  number  = {1},
  pages   = {219--255},
  year    = {2004}
}

@unpublished{kim2026,
  author  = {Jae Kwang Kim and Yonghyun Kwon and Yumou Qiu},
  title   = {Bregman projection for calibration estimation in survey sampling},
  year    = {2026},
  note    = {submitted (https://arxiv.org/abs/2603.20780)}
}

@article{Chernozhukov2018,
  author  = {Victor Chernozhukov and Denis Chetverikov and Mert Demirer and Esther Duflo and Christian Hansen and Whitney Newey and James Robins},
  title   = {Double/debiased machine learning for treatment and structural parameters},
  journal = {The Econometrics Journal},
  volume  = {21},
  number  = {1},
  pages   = {C1--C68},
  year    = {2018}
}

@article{Basu1998,
  author  = {Ayanendranath Basu and Ian R. Harris and Nils L. Hjort and M. C. Jones},
  title   = {Robust and efficient estimation by minimising a density power divergence},
  journal = {Biometrika},
  volume  = {85},
  number  = {3},
  pages   = {549--559},
  year    = {1998}
}

@article{WarwickJones2005,
  author  = {Janette Warwick and M. Chris Jones},
  title   = {Choosing a robustness tuning parameter},
  journal = {Journal of Statistical Computation and Simulation},
  volume  = {75},
  number  = {7},
  pages   = {581--588},
  year    = {2005}
}

@article{BasakBasuJones2021,
  author  = {Sancharee Basak and Ayanendranath Basu and M. C. Jones},
  title   = {On the `optimal' density power divergence tuning parameter},
  journal = {Journal of Applied Statistics},
  volume  = {48},
  number  = {3},
  pages   = {536--556},
  year    = {2021}
}

@article{Jaynes1957,
  author  = {Edwin T. Jaynes},
  title   = {Information theory and statistical mechanics},
  journal = {Physical Review},
  volume  = {106},
  number  = {4},
  pages   = {620--630},
  year    = {1957}
}

@article{Phillips2006,
  author  = {Steven J. Phillips and Robert P. Anderson and Robert E. Schapire},
  title   = {Maximum entropy modeling of species geographic distributions},
  journal = {Ecological Modelling},
  volume  = {190},
  number  = {3--4},
  pages   = {231--259},
  year    = {2006},
  doi     = {10.1016/j.ecolmodel.2005.03.026}
}

@article{Elith2011,
  author  = {Jane Elith and Steven J. Phillips and Trevor Hastie and Miroslav Dud\'ik and Yung En Chee and Colin J. Yates},
  title   = {A statistical explanation of {MaxEnt} for ecologists},
  journal = {Diversity and Distributions},
  volume  = {17},
  number  = {1},
  pages   = {43--57},
  year    = {2011},
  doi     = {10.1111/j.1472-4642.2010.00725.x}
}

@article{RennerWarton2013,
  author  = {Ian W. Renner and David I. Warton},
  title   = {Equivalence of {MAXENT} and {P}oisson point process models for species distribution modeling in ecology},
  journal = {Biometrics},
  volume  = {69},
  number  = {1},
  pages   = {274--281},
  year    = {2013},
  doi     = {10.1111/j.1541-0420.2012.01824.x}
}

@book{Harte2011,
  author    = {John Harte},
  title     = {Maximum Entropy and Ecology: A Theory of Abundance, Distribution, and Energetics},
  publisher = {Oxford University Press},
  year      = {2011}
}

@article{GneitingRaftery2007,
  author  = {Tilmann Gneiting and Adrian E. Raftery},
  title   = {Strictly proper scoring rules, prediction, and estimation},
  journal = {Journal of the American Statistical Association},
  volume  = {102},
  number  = {477},
  pages   = {359--378},
  year    = {2007}
}

@article{GrunwaldDawid2004,
  author  = {Peter D. Gr{\"u}nwald and A. Philip Dawid},
  title   = {Game theory, maximum entropy, minimum discrepancy and robust {B}ayesian decision theory},
  journal = {The Annals of Statistics},
  volume  = {32},
  number  = {4},
  pages   = {1367--1433},
  year    = {2004}
}

@article{HorvitzThompson1952,
  author  = {Daniel G. Horvitz and Donovan J. Thompson},
  title   = {A generalization of sampling without replacement from a finite universe},
  journal = {Journal of the American Statistical Association},
  volume  = {47},
  number  = {260},
  pages   = {663--685},
  year    = {1952}
}

@inproceedings{Minka2001,
  author    = {Thomas P. Minka},
  title     = {Expectation propagation for approximate {B}ayesian inference},
  booktitle = {Proceedings of the Seventeenth Conference on Uncertainty in Artificial Intelligence},
  pages     = {362--369},
  year      = {2001}
}

@techreport{Minka2005,
  author      = {Thomas P. Minka},
  title       = {Divergence measures and message passing},
  institution = {Microsoft Research},
  number      = {MSR-TR-2005-173},
  year        = {2005}
}

@unpublished{briol2019,
    author = 
{Briol, F.-X. and Barp, A. and  Duncan, A. B. and Girolami, M.},
 year = {2019}, 
 title = {Statistical inference for generative models with maximum mean discrepancy}, 
 note = {arXiv:1906.05944}
 }

\end{document}